\documentclass[12pt]{amsart}

\usepackage{amsfonts,amssymb,enumerate,color,bm,hhline,makecell,array}
\usepackage{array}
\usepackage{mathrsfs}
\usepackage{comment}
\usepackage[all,ps,cmtip]{xy} 
\usepackage[normalem]{ulem}
\usepackage[colorlinks=true,citecolor=blue, urlcolor=blue, linkcolor=blue, pagebackref]{hyperref}
\usepackage{tikz-cd}
\usepackage{amscd}
\usepackage[shortlabels]{enumitem}
\usepackage{tensor}
\usepackage{tikz}
\usetikzlibrary{matrix,arrows}
\usepackage{extarrows}
\usepackage{subcaption}

\def\C{\ensuremath{\mathbb{C}}}

\def\H{\ensuremath{\mathbb{H}}}

\def\P{\ensuremath{\mathbb{P}}}
\def\Q{\ensuremath{\mathbb{Q}}}
\def\R{\ensuremath{\mathbb{R}}}
\def\Z{\ensuremath{\mathbb{Z}}}

\def\cA{\ensuremath{\mathcal A}}

\def\cC{\ensuremath{\mathcal C}}

\def\cE{\ensuremath{\mathcal E}}
\def\cF{\ensuremath{\mathcal F}}

\def\cH{\ensuremath{\mathcal H}}
\def\cK{\ensuremath{\mathcal K}}
\def\cL{\ensuremath{\mathcal L}}

\def\cN{\ensuremath{\mathcal N}}
\def\cO{\ensuremath{\mathcal O}}
\def\cP{\ensuremath{\mathcal P}}
\def\cQ{\ensuremath{\mathcal Q}}
\def\cT{\ensuremath{\mathcal T}}
\def\cW{\ensuremath{\mathcal W}}
\def\cX{\ensuremath{\mathcal X}}

\def\ss{\ensuremath{\mathbf s}}
\def\tt{\ensuremath{\mathbf t}}
\def\uu{\ensuremath{\mathbf u}}
\def\vv{\ensuremath{\mathbf v}}
\def\ww{\ensuremath{\mathbf w}}
\def\ff{\ensuremath{\mathbf f}}
\def\ee{\ensuremath{\mathbf e}}
\def\xx{\ensuremath{\mathbf x}}
\def\rr{\ensuremath{\mathbf r}}

\def\fM{\mathfrak M}

\def\phi{{\varphi}}

\DeclareMathOperator{\Aut}{Aut}

\DeclareMathOperator{\ch}{ch}

\DeclareMathOperator{\Ext}{Ext}

\DeclareMathOperator{\Hilb}{Hilb}
\DeclareMathOperator{\Hom}{Hom}

\DeclareMathOperator{\NS}{NS}

\DeclareMathOperator{\Pic}{Pic}

\DeclareMathOperator{\Stab}{Stab}

\DeclareMathOperator{\td}{td}

\newcommand{\RcHom}{\mathrm{R}\mathcal{H}\!{\it om}}

\def\Aut{\mathop{\mathrm{Aut}}\nolimits}

\def\ch{\mathop{\mathrm{ch}}\nolimits}

\def\Coh{\mathop{\mathrm{Coh}}\nolimits}

\def\ext{\mathop{\mathrm{ext}}\nolimits}
\def\Ext{\mathop{\mathrm{Ext}}\nolimits}
\def\Hilb{\mathop{\mathrm{Hilb}}\nolimits}

\def\hom{\mathop{\mathrm{hom}}\nolimits}
\def\Hom{\mathop{\mathrm{Hom}}\nolimits}

\def\RHom{\mathop{\mathrm{RHom}}\nolimits}
\def\id{\mathop{\mathrm{id}}\nolimits}

\def\NS{\mathop{\mathrm{NS}}\nolimits}

\def\Pic{\mathop{\mathrm{Pic}}\nolimits}

\def\td{\mathop{\mathrm{td}}\nolimits}

\def\Stab{\mathop{\mathrm{Stab}}\nolimits}

\def\ST{\mathrm{ST}}
\def\norm#1{\left\|#1\right\|}

\def\Halg{{H^{*}_{\mathrm{alg}}(X,\mathbb{Z})}}

\def\onto{\ensuremath{\twoheadrightarrow}}

\newtheorem{Thm}{Theorem}[section]
\newtheorem{Prop}[Thm]{Proposition}

\newtheorem{Lem}[Thm]{Lemma}
\newtheorem{Cor}[Thm]{Corollary}

\newtheorem*{Ques*}{Question}

\theoremstyle{definition}
\newtheorem{Defi}[Thm]{Definition}
\newtheorem{Rem}[Thm]{Remark}
\newtheorem{Ex}[Thm]{Example}

\makeatletter
\@namedef{subjclassname@2020}{%
  \textup{2020} Mathematics Subject Classification}
\makeatother

\makeatletter
\def\@tocline#1#2#3#4#5#6#7{\relax
  \ifnum #1>\c@tocdepth % then omit
  \else
    \par \addpenalty\@secpenalty\addvspace{#2}%
    \begingroup \hyphenpenalty\@M
    \@ifempty{#4}{%
      \@tempdima\csname r@tocindent\number#1\endcsname\relax
    }{%
      \@tempdima#4\relax
    }%
    \parindent\z@ \leftskip#3\relax \advance\leftskip\@tempdima\relax
    \rightskip\@pnumwidth plus4em \parfillskip-\@pnumwidth
    #5\leavevmode\hskip-\@tempdima
      \ifcase #1
       \or\or \hskip 1em \or \hskip 2em \else \hskip 3em \fi%
      #6\nobreak\relax
    \dotfill\hbox to\@pnumwidth{\@tocpagenum{#7}}\par
    \nobreak
    \endgroup
  \fi}
\makeatother

\allowdisplaybreaks

\usepackage{enumitem}
\setlist[itemize]{noitemsep,nolistsep}
\setlist[enumerate]{noitemsep,nolistsep}
\usepackage{colortbl}
\usepackage{enumerate}

\begin{document}

\definecolor{xdxdff}{rgb}{0.49019607843137253,0.49019607843137253,1}
\definecolor{uuuuuu}{rgb}{0.26666666666666666,0.26666666666666666,0.26666666666666666}
\definecolor{ffqqqq}{rgb}{1,0,0}

\title{Stable Sheaves on K3 Surfaces via Wall-Crossing}

\author[Alessio Bottini]{Alessio Bottini}

\address{Dipartimento di Matematica, Universit\`{a} di Roma Tor Vergata, Via della Ricerca Scientifica 1, 00133, Roma, Italia}
\address{Universit\'e Paris-Saclay,
CNRS, Laboratoire de Math\'ematiques d'Orsay,
Rue Michel Magat, B\^at. 307, 91405 Orsay, France}
\email{bottini@mat.uniroma2.it}

%MathSubjClass2020
\makeatletter
\@namedef{subjclassname@2020}{%
  \textup{2020} Mathematics Subject Classification}
\makeatother

\begin{abstract}
We give a new proof of the following theorem: moduli spaces of stable complexes on a complex projective K3 surface, with primitive Mukai vector and with respect to a generic Bridgeland stability condition, are hyperk\"{a}hler varieties of $\mathrm{K3}^{[n]}$-type of expected dimension. We use derived equivalences, deformations and wall-crossing for Bridgeland stability to reduce to the case of the Hilbert scheme of points.
\end{abstract}

\maketitle
\setcounter{tocdepth}{1}
\tableofcontents

%%%%%%%%%%%%%%%%%%%%%%%%%%%%%%%%%%%%%

\section{Introduction}\label{sec:intro}
Moduli spaces of semistable sheaves on a complex projective K3 surface $X$ are a well studied class of algebraic varieties, and they are among the only known examples of compact hyperk\"{a}hler (or irreducible holomorphic symplectic) varieties. Classically, we consider the moduli space $M_H(\vv)$ of Gieseker-stable coherent sheaves with fixed topological invariants, encoded in the Mukai vector $\vv \in \Halg$. Recall that, given a polarization $H$, a coherent sheaf $E$ is Gieseker semistable if it is pure and 
\[ 
   p(F,m) \leq p(E,m)
\]
for every proper subsheaf $F \subset E$, where $p(E,m)$ is the reduced Hilbert polynomial. It is stable if the strict inequality holds. 
The stability condition gives a GIT construction of $M_H(\vv)$, which is then projective. If $\vv$ is primitive, and $H$ is generic, i.e. it is in the complement of a union of hyperplanes in the ample cone, then $M_H(\vv)$ is smooth and parameterizes stable sheaves. 

In \cite{Bridgeland_triangulated} and \cite{bridgeland_stability_2006} Bridgeland defined the notion of a stability condition on a triangulated category, and constructed stability conditions on the bounded derived category of coherent sheaves $D^b(X)$ on a K3 surface $X$. These stability conditions form a complex manifold $\Stab(X)$, and given a Mukai vector $\vv \in \Halg$ there is a set of real-codimension one submanifolds, such that stability of objects with class $\vv$ is constant in each chamber, i.e. in each connected component of the complement of the walls. If $\vv$ is primitive, we say that a stability condition $\sigma \in \Stab(X)$ is $\vv$-generic if it varies in a chamber for $\vv$. In this case, every $\sigma$-semistable object is $\sigma$-stable. Moreover, there is a chamber, near the ``large volume limit", where Bridgeland stability recovers Gieseker stability. The connected component of $\Stab(X)$ containing this chamber is called the \textit{distinguished component} and denoted by $\Stab^{\dagger}(X)$. 

Moduli stacks of semistable complexes were studied by many people: Toda \cite{Toda_Moduli}, Abramovich-Polishchuk \cite{AP:tstructures}, and finally a complete treatment can be found \cite[Part II]{bayer2019stability}. If $\vv$ is primitive, and $\sigma \in \Stab^{\dagger}(X)$ is $\vv$-generic there exists a coarse moduli space as an algebraic space, and it parameterizes stable complexes.
Moreover, under these assumptions the coarse moduli space is a smooth and proper algebraic space, by results of Inaba \cite{inaba_stable,inaba_smoothness_2010} and Lieblich \cite{Lieblich_complexes}. In contrast to the classical Gieseker moduli spaces, these have no obvious GIT construction. Projectivity was shown in \cite{bayer_projectivity_2013}. The idea is to use a Fourier-Mukai argument to reduce to the classical case of (twisted) Gieseker stability for which a GIT construction is available.

The goal of this paper is to give a new proof of the following result.

\begin{Thm}[Main Theorem]
Let $X$ be a K3 surface. Let $\Halg$ be its extended Mukai lattice, together with the Mukai Hodge structure.
Assume that $\vv\in \Halg$ is a primitive vector and let $\sigma\in\Stab^\dagger(X)$ be a $\vv$-generic stability condition on $X$.
Then:
\begin{enumerate}[{\rm (1)}]
\item The moduli space $M_{\sigma}(\vv)$ is non-empty if and only if $\vv^2\geq-2$.
Moreover, it is a smooth projective hyperk\"{a}hler variety of dimension $\vv^2 + 2$, deformation-equivalent to a Hilbert scheme of points on a K3 surface.
\item If $\vv^2\geq 0$, then there exists a natural Hodge isometry 
\[
\theta_{\vv}^{-1}\colon H^2(M_\sigma(\vv),\Z)\xrightarrow{\quad\sim\quad}
\begin{cases}\vv^\perp & \text{if }\vv^2>0\\ \vv^\perp/\Z\vv & \text{if } \vv^2=0,\end{cases}
\]
where the orthogonal is taken in $H^{*}(X,\mathbb{Z})$.
\end{enumerate}
\end{Thm}
Here $\theta_{\vv}$ is the \textit{Mukai homomorphism}, see Definition \ref{Def:MukaiHom} for the precise definition. 

The analogous result for Gieseker stable sheaf is a celebrated theorem due to the work of many people: Beauville \cite{beauville_1983}, Mukai \cite{mukai_symplectic_1984,Muk:K3}, Kuleshov \cite{kuleshov90}, O'Grady \cite{ogrady_weight-two_1995}, Huybrechts \cite{huybrechts_birational}, Yoshioka \cite{Yoshioka_main}, and others. The complete proof is in \cite{Yoshioka_main}, a recent account and a generalization to the non-primitive case is given in \cite{perego_moduli_2018}. Thanks to the work of Toda \cite{Toda_Moduli} and Bayer and Macr\`i \cite{bayer_projectivity_2013} the classical result is enough to deduce the Main Theorem for moduli spaces of Bridgeland stable complexes. Nevertheless, we feel that giving a complete proof from scratch in this generality is conceptually better. For example, in our argument there is no need to study preservation of Gieseker stability under Fourier-Mukai. This was a difficult technical step in the classical proof, and was investigated by several people \cite{bartocci_1997}, \cite{Muk:FM}, \cite{Yoshioka_FM1,Yoshioka_FM2}. This step gets simplified by Bridgeland stability conditions: we only need to check that the equivalences we use preserve the distinguished component $\Stab^{\dagger}(X)$. Of course, since Gieseker stability can be recovered via Bridgeland stability in the large volume limit, the classical result follows from the Main Theorem.

Now we briefly explain the idea of the proof, the complete argument will be in Section \ref{sec:Deformation}. We start with a K3 surface $X$, a primitive Mukai vector $\vv \in \Halg$ with $\vv^2 \geq -2$ and a generic stability condition $\sigma$ in the distinguished component $\Stab^{\dagger}(X)$. The basic geometric input is that the Main Theorem holds for Hilbert schemes of points on a K3 surface, this is a classical result due to Beauville \cite{beauville_1983}. So, we want to reduce from our starting moduli space $M_{\sigma}(\vv)$ on $X$, to the Hilbert scheme $\Hilb^n(Y)$ on another K3 surface $Y$, in such a way that the Main Theorem remains true at every step.
For this reduction we will use the following tools:
\begin{enumerate}
    \item Derived equivalences: 
    \begin{itemize}
        \item Shifts $E \mapsto E[1]$,
        \item Tensor product with $L \in \Pic(X)$,
        \item The spherical twist $\ST_{\cO_X}$ around the structure sheaf,
        \item The Fourier-Mukai transform $\Phi_{\cE}$ with kernel the universal family of a fine, two-dimensional, projective moduli space parameterizing Gieseker-stable sheaves. 
    \end{itemize}  
    \item Existence of relative stability conditions on a smooth projective family $\cX \rightarrow C$ of K3 surfaces over a smooth quasi-projective curve, and existence of the corresponding relative moduli spaces, this is done in \cite{bayer2019stability}.    
    \item Wall-crossing for moduli spaces of spherical and isotorpic objects on any K3 surface, and for the Hilbert scheme $\Hilb^n(X)$ on a K3 surface $X$ with $\Pic(X)=\Z \cdot H$ with $H^2=2k^2(n-1)$.
\end{enumerate}
In the first five sections we show that the statement of the Main Theorem is invariant under operations of the above type, and in Section \ref{sec:Deformation} we combine them to conclude the argument. 

The argument goes roughly as follows: we begin with a sequence of autoquivalences of type $(1)$ to modify the Mukai vector $\vv$. This is done for the following reason. We can choose a polarization $H$ on $X$ with $H^2=2d$ and the new Mukai vector $\vv'$, so that its Hodge locus in the moduli space of polarized K3 surfaces of degree $2d$ contains a polarized K3 surface $(Y,H')$ with the following properties: 
\begin{enumerate}
    \item Its Picard group $\Pic(Y)$ is an hyperbolic plane.
    \item There is an algebraic class $\ww \in H_{\mathrm{alg}}^*(Y,\Z)$ such that the moduli space $M:=M_{H'}(\ww)$ is fine, non-empty, and a K3 surface. 
    \item The product of the classes $\vv$ and $\ww$ is $(\vv,\ww)=-1$.
\end{enumerate} 
Then, we deform to this K3 surface $Y$, and consider the Fourier-Mukai transform given by the universal family $\cE$ of $M$. The transformed vector is, up to tensoring with line bundles on $M$, the Mukai vector of the Hilbert scheme. 

If $\vv^2=0$ or $-2$ we argue as follows. We connect the resulting stability condition on $M$ to the Gieseker chamber with a path. This path meets finitely many walls, so we only need to study wall-crossing at each of them. For this, we prove the following result.
\begin{Thm}[Theorem \ref{semirigid1}]
Let $X$ be a K3 surface, $\vv$ be a primitive vector, with $\vv^{2}=-2$ or $\vv^2=0$. Let $\cW$ be a wall for the wall and chamber decomposition for $\vv$, and denote by $\sigma_{\pm}$ two generic stability conditions, one on each side of the wall, and $\sigma_0 \in \cW$ a generic stability condition on the wall.
\begin{itemize}
    \item If $\vv^{2}=-2$, then $M_{\sigma_{+}}(\vv) \neq \emptyset $ implies $M_{\sigma_{-}}(\vv) \neq \emptyset$. 
    \item If $\vv^{2}=0$, then there exists a spherical, $\sigma_{0}$-stable object $S$ such that either $\ST_{S}: M_{\sigma_{+}} \rightarrow M_{\sigma_{-}}$ or $\ST_{S}^{\pm 2}: M_{\sigma_{+}} \rightarrow M_{\sigma_{-}}$ are isomorphism.
\end{itemize}
\end{Thm}
By applying the above Theorem finitely many times we  complete the proof of the Main Theorem in the case where $\vv^2=0$ or $-2$.

A similar statement holds for $\vv^2>0$ but is much more complicated, and it is the main result of \cite{mmp}. In general, wall-crossing induces birational maps $M_{\sigma_+}(\vv) \dashrightarrow M_{\sigma_-}(\vv)$, which are not necessarily isomorphisms. Moreover, showing that this birational map is defined in codimension one requires a detailed analysis, and it relies on existence of stable complexes. 

Instead, for the positive square case we use a different argument. We note that the Picard group $\Pic(M)$ of the Fourier-Mukai partner $M$ is again an hyperbolic plane, hence is has polarizations of any degree. In particular, we can deform to a K3 surface $X'$ with $\Pic(X')=\Z \cdot H'$ with $(H')^2=2k^2(n-1)$. Up to changing the Fourier-Mukai partner $M$ with an isomorphic one obtained by wall-crossing via the above theorem, we can assume that the resulting stability condition on $X'$ lies in a domain $V(X') \subset \Stab^{\dagger}(X')$. This can be characterized as the locus of stability conditions where all the skyscraper sheaves are stable of phase one, see Definition \ref{def:V(X)} and Lemma \ref{UandV}. Under these assumptions, the Main Theorem will be established in Section \ref{subsec:Positivesquare}, where we prove the following result.

\begin{Thm}[Corollary \ref{Cor:unigonal}]
Let $X$ be a K3 surface with $Pic(X)=\Z \cdot H$ with $H^2=2d$ and $d=k^2(n-1)$ for $k>1$ integer. There is only one wall for $\vv=(1,0,1-n)$ in $V(X)$, and the shifted derived dual $\RcHom(-,\cO_X)[2]$ induces an isomorphism
\[ M_{\sigma_{+}}(\vv) \xrightarrow{\sim} M_{\sigma_{-}}(\vv),\]
where $\sigma_+$ and $\sigma_-$ are two generic stability conditions in the two chambers. In particular, the Main Theorem holds for both of them.
\end{Thm}

\subsection*{Structure of the paper}
In Section \ref{sec:bridgelandstability} we review the theory of Bridgeland stability conditions on K3 surfaces. We pay particular attention to those results in \cite{bridgeland_stability_2006} which rely on existence of stable sheaves. The main purpose of this section is to recall the definition of the distinguished component $\Stab^{\dagger}(X)$ of stability conditions. We also recall a result by Hartmann: the derived equivalences above preserve the distinguished component $\Stab^{\dagger}(X)$. 

In Section \ref{sec:geometry} we review some aspects of theory of moduli of stable complexes, and hyperk\"{a}hler varities. In Section \ref{sec:Wall crossing} we study the wall-crossing behavior for Mukai vectors $\vv$ with $\vv^2 \leq 0$, and prove Theorem \ref{semirigid1}. In Section \ref{subsec:Positivesquare} we study wall-crossing for the Hilbert scheme on a K3 surface with $\Pic(X)=\Z\cdot H$ and degree $H^2=2k^2(n-1)$. In this section we restrict our attention to stability conditions of the form $\sigma_{\alpha H,\beta H}$, with $\alpha>0$. In Section \ref{sec:Deformation} we complete the proof of the Main Theorem, by reducing to the case of the Hilbert scheme. 

\subsection*{Acknowledgments}
I want to thanks my advisors Emanuele Macr\`{i} and Antonio Rapagnetta for suggesting the problem and for many useful and instructive discussions. I am partially supported by the MIUR Excellence Department Project awarded to the Department of Mathematics, University of Rome Tor Vergata, CUP E83C18000100006 and the ERC Synergy Grant ERC- 2020-SyG-854361-HyperK.

\section{Review: Bridgeland stability conditions} 
\label{sec:bridgelandstability}
In this section we review the theory of Bridgeland stability on K3 surfaces, as introduced in \cite{bridgeland_stability_2006}. The main objective is to define the \textit{distinguished component} $\Stab^{\dagger}(X)$ of the space of stability conditions and to show that (some) derived equivalences preserve this component. All of the results here are well known, due to Bridgeland and Hartmann, but some proofs in the literature use existence of slope stable spherical sheaves. Here we give a treatment that avoids that problem by slightly modifying the standard definitions. 

\subsection{Basic definitions and results}
Let $X$ be a K3 surface, denote by $D^b(X)$ the bounded derived category of coherent sheaves on $X$, and by \[\Halg:=H^0(X,\Z) \oplus \NS(X) \oplus H^4(X,\Z)\] the algebraic part of the cohomology. It comes equipped with an integral even bilinear form of signature $(2,\rho(X))$, called \textit{Mukai pairing} and defined by:
\[ (\vv,\vv')=\Delta.\Delta' - rs'-r's, \]
where we write $\vv=(r,\Delta,s)$ and $\vv'=(r',\Delta',s')$. Recall that given an object $E \in D^b(X)$, its \textit{Mukai vector} $v(E) \in \Halg$ is defined as \[v(E):=\ch(E).\sqrt{\td_X}=(\ch_0(E),\ch_1(E),\ch_2(E)+\ch_0(E)).\]

\begin{Defi}
A (full, numerical) stability condition is a pair $\sigma=(Z,\cA)$, where $Z: \Halg \rightarrow \C$ is a group homomorphism (called \textit{central charge}), and $\cA \subset D^b(X)$ is a heart of a bounded t-structure, satisfying the following properties:
\begin{enumerate}
    \item For any $0 \neq E \in D^b(X)$, the central charge $Z(E)$ lies in the semi-closed upper half-plane
    \[ \H:=\R_{>0}e^{(0,1]i\pi} \]
    \item Given an object $0 \neq E \in \cA$ we define the slope as  $\nu_{\sigma}(E):=\frac{-\Re Z(E)}{\Im Z(E)}$. A non-zero object $E$ is said (semi)stable if for every proper subobject $F \subset E$ the following inequality holds:
    \[ \nu_{\sigma}(F) < (\leq)  \nu_{\sigma}(E). \] 
    Then, every $0 \neq E\in \cA$ has a Harder-Narasimhan filtration, i.e. a filtration
    \[ 0=E_0 \subset E_1 \subset \dots \subset E_n=E, \]
    with semistable quotients of decreasing slope.
    \item Fix a norm $\norm *$ on $\Halg$. Then there is a constant $C >0$, such that for every semistable object $E \in \cA$, we have
    \[ \norm E < C|Z(E)|\]
\end{enumerate}
\end{Defi}

Given a pair $(Z,\cA)$ as above, we can extend the notion of stability to the full derived category $D^b(X)$ in the following way. For every $\phi \in (0,1]$ define $\cP(\phi)$ as the full subcategory of semistable objects $E \in \cA$ with $Z(E) \in \R_{>0}e^{i\phi}$. Then extend this definition to every $\phi \in \R$ by the compatibility condition $\cP(\phi+n)=\cP(\phi)[n]$. 

Every non-zero object $E \in D^b(X)$ has a Harder-Narasimhan filtration, i.e. a sequence of maps
\[0=E_{0} \rightarrow E_1 \rightarrow \dots \rightarrow E_{n-1}\rightarrow E_{n}=E, \]
with cones $A_i$ that are semistable of decreasing phases.
The phases of the first and last Harder-Narasimhan factors are denoted by $\phi_{\sigma}^+(E)$ and $\phi_{\sigma}^-(E)$. 
The category $\cP(\phi)$ is abelian of finite lenght, so every semistable object has a Jordan-Holder filtration, i.e. a finite filtration with stable cones of the same phase. Two semistable objects with the same associated graded are called $S$-equivalent.

It is shown in \cite[Proposition $5.3$]{Bridgeland_triangulated} that the data of $(Z,\{\cP(\phi)\}_{\phi \in \R})$ is equivalent to the data of the heart $\cA$ and the central charge $Z$. The inverse equivalence is given by forming the category $\cP(0,1]$, where $\cP((a,b])$ is the full subcategory of $D^b(X)$ with objects 
\[ \{ E \in D^b(X) \mid \phi^{-}(E),\phi^+(E) \in (a,b]\}.\]

The definition was extended in \cite{bayer2019stability} to include openness of stability in families and existence of moduli spaces. Precisely we add the conditions:
\begin{enumerate}
    \item[(4)] For every scheme $T$ and for every $T$-perfect complex $E \in D_{T-\textrm{perf}}(S \times T)$ the set 
    \[ 
    \{ t \in T \mid E_t \in \cP(\phi) \}
    \]
    is open.
    \item[(5)] for every $\vv \in \Halg$ and every $\phi \in \R$ such that $Z(\vv) \in \R_{>0}e^{i\pi\phi}$ the functor
    \[
    T \rightarrow \mathfrak{M}_{\sigma}(\vv,\phi):=\{ E \in D_{T-\textrm{perf}}(S \times T) \mid E_t \in \cP(\phi) \textrm{ and } v(E_t)=\vv \}
    \]
    is bounded. 
\end{enumerate}

Let $\Stab(X)$ be the set of all stability condition. It has a natural topology induced by a metric, see \cite[Section $6$]{Bridgeland_triangulated} for the precise form of the metric. This topology can be characterized as the coarsest topology that makes the functions \[\sigma \mapsto Z \ \mathrm{and} \  \sigma \mapsto \phi_{\sigma}^{\pm}(E)\] continuous, for every $0 \neq E \in D^b(X)$.
The main result in \cite{Bridgeland_triangulated} is the following. 

\begin{Thm}[Bridgeland Deformation Theorem]
The map
\begin{align*}
    \pi:\Stab(X) &\rightarrow \Hom(\Halg,\C)\\
    \sigma=(Z,\cA)& \mapsto Z 
\end{align*}
is a local homeomorphism. In particular, every connected component of $\Stab(X)$ is a complex manifold of dimension $\mathrm{rk}(\Halg)$.
\end{Thm}

\begin{Rem}[{{\cite[Lemma $8.2$]{Bridgeland_triangulated}}}]
There are two natural actions on the space of stability conditon: a left action by the group $\Aut(D^b(X))$ of exact autoequivalences of $D^b(X)$, and a right action by $\widetilde{GL}^+_2(\R)$, the universal cover of the group ${GL}^+_2(\R)$. Given an autoequivalence $\Phi$ and a stability condition $\sigma=(Z,\cP)$ we set $\Phi(\sigma):=(Z \circ \Phi^{-1},\cP')$, where $\cP'(t):=\Phi(\cP(t))$. The action of $\widetilde{GL}^+_2(\R)$ is given by lifting the right action of $GL_2^+(\R)$ on $\Hom(\Halg,\C)$.
\end{Rem}

We are interested in knowing how stability varies when we deform $\sigma$; this was first done in \cite[Proposition $9.3$]{bridgeland_stability_2006}. See also \cite[Proposition $3.3$]{bayer_local} and \cite[Proposition $2.8$]{Toda_Moduli}.

\begin{Prop}\label{walls}
Fix a class $\vv \in \Halg$. There exists a locally finite set of real codimension one manifold with boundary, called walls, in $\Stab(X)$, such that when $\sigma$ varies within a chamber (a connected component of the complement of the set of walls), the set of $\sigma$-semistable and $\sigma$-stable objects does not change. If $\vv$ is primitive and $\sigma$ varies in a chamber, every semistable object is stable.
\end{Prop}

\begin{Defi}
Let $\vv \in \Halg$. We say a stability condition $\sigma$ is $\vv$ generic if is not on a wall for $\vv$. We say a stability condition is generic on a wall if it lies on only one wall.
\end{Defi}

\begin{Rem}
If an object $E$ is semistable in a chamber, it continues to be semistable on a wall of that chamber. Indeed, the condition for $E$ to be semistable is given by the equality $\phi^{+}_{\sigma}(E)=\phi^{-}_{\sigma}(E)$, which is a closed condition on the space of stability conditions. 
\end{Rem}

\subsection{Construction of stability conditions on K3 surfaces.}\label{subsec:stabilityK3} We review the construction of stability conditions on K3 surfaces, given in \cite{bridgeland_stability_2006}. 
Let $\omega \in NS(X) \otimes \R$ ample.
Recall the definition of the slope of a coherent sheaf $E \in \Coh(X)$, 
\[ \mu_{\omega}(E):=
\begin{cases}
\frac{\omega.c_1(E)}{\omega^2r(E)}  & \text{if } r(E)>0,\\
+\infty & \mathrm{otherwise},
\end{cases}
\]
where $r(E)$ denotes the rank of $E$.

\begin{Defi}%[{{\cite[Definition $2.1$]{bayer_wallcross_bn}}}]
A coherent sheaf $E \in \Coh(X)$ is slope (semi)stable if for every proper subsheaf $A \subset E$ we have
\[ \mu_{\omega}(A) < (\leq)  \mu_{\omega}(E/A). \]
\end{Defi}

For later use we also recall the definition of $B$-twisted Gieseker stability, for $B$ a rational divisor. Note that if $B=0$ we recover the classical notion of Gieseker stability. 

\begin{Defi}
Let $\omega, B \in \NS(X)_{\Q}$, with $\omega$ ample. We define the $B$-twisted Hilbert polynomial of a coherent sheaf $E$ as
\[ P(E,m):=\int_X{e^{m\omega-B}}.v(E). \]
A pure $d$-dimensional coherent sheaf $E$ is $B$-twisted Gieseker (semi)stable if, for every proper non trivial subsheaf $F \subset E$ we have
\[ \frac{P(F,m)}{\alpha_d(F)} < (\leq) \frac{P(E,m)}{\alpha_d(E)}, \]
for $m >>0$, where $\alpha_d(E)$ is the degree $d$ coefficient of $P(E,m)$. 
\end{Defi}

\begin{Rem}
Similarly to Bridgeland stability, both slope stability and Gieseker stability satisfy the existence of Harder-Narasimhan filtrations. That is every non-zero coherent sheaf $E \in \Coh(X)$ has a filtration with slope-semistable (resp. Gieseker semistable) quotients of decreasing slope (resp. decreasing reduced Hilbert polynomial).
\end{Rem}

Now consider the pair $\sigma_{\omega,B}=(Z_{\omega,B},\Coh^{\omega,B})$ where
\[ Z_{\omega,B}(\vv):=(\vv,\exp(B+i\omega)),\]
and $\Coh^{\omega,B}(X)$ is the tilt of $\Coh(X)$ with respect to the torsion pair $(\cT^{\beta},\cF^{\beta})$, defined as follows
\begin{align*}
    \cT^{\beta}&:=\{ T \in \Coh(X) \mid \textrm{All HN factors $A_{i}$ of $T$ satisfy}\  \mu_{\omega}(A_{i})>\frac{\omega.B}{\omega^{2}} \};\\
    \cF^{\beta}&:=\{ F \in \Coh(X) \mid \textrm{All HN factors $A_{i}$ of $F$ satisfy} \  \mu_{\omega}(A_i) \leq \frac{\omega.B}{\omega^2} \},
\end{align*}
where the Harder-Narasimhan factors are with respect to slope stability.

\begin{Defi}
An object $E \in D^b(X)$ is spherical if 
\[
\Ext^i(E,E)=
\begin{cases}
\C & \text{if } i=0,2 \\
0 & \text{otherwise.}
\end{cases}
\]
\end{Defi}

\begin{Thm}[{{\cite[Lemma $6.2$]{bridgeland_stability_2006}}}]
The above construction gives a stability condition $\sigma_{\omega,B}$ on $D^b(X)$, provided $Z_{\omega,B}(E) \not \in \R_{\leq 0}$ for every spherical torsion-free sheaf $E$. 
\end{Thm}

\begin{Defi}\label{def:V(X)}
The set of stability conditions $\sigma_{\omega,B}=(Z_{\omega,B},\Coh^{\omega,B})$, with $Z_{\omega,B}(E) \not \in \R_{\leq 0}$ for every spherical sheaf $E$ is denoted by $V(X)$. We define the \textit{geometric chamber} $U(X)$ as the subset \[\widetilde{GL_{2}}^{+}(\R).V(X) \subseteq \Stab(X)\] obtained from $V(X)$ via the action of $\widetilde{GL_{2}}^{+}(\R)$. A stability condition is \textit{geometric} if it belongs to $U(X)$.
\end{Defi}

To define the distinguished component we need to show that the sets $V(X)$ and $U(X)$ are connected. To show this we follow the proof in \cite{bridgeland_stability_2006} and introduce several auxiliary spaces. Using the Mukai pairing form on $\Halg$ we can identify the central charge $Z$ of a stability condition with a vector $\Omega_Z \in \Halg \otimes \C$. 

Define $\cP(X) \subset \Halg$ as the set of vectors $\Omega$ whose real and imaginary parts span positive definite two-planes in $\Halg \otimes \R$. Define $\cP_0(X)$ as the subset of $\cP(X)$ of classes not orthogonal to any spherical class:
\[ \cP_0(X):=\{ \Omega \in \cP(X) \mid (\Omega,\delta) \neq 0 \textrm{ for every } \delta \in \Delta(X)\}\]
where $\Delta(X):=\{\delta \in \Halg \mid \delta^2=-2\}$.
Consider the subset of $\cP_0(X)$ given by the vectors $\Omega$ obtained by the construction above:
\[ \cK(X):=\{ \Omega \in \cP_0(X) \mid \Omega=\exp(B+i\omega) \textrm{ with } \omega \in \mathrm{Amp}(X) \textrm{ and } B \in \NS(X) \otimes \R \} \]
The set $\cP_0(X)$ has two connected components, we call $\cP^+_0(X)$ the one containing $\cK(X)$. 

\begin{Prop}[{{\cite[Proposition $8.3$]{bridgeland_stability_2006}}}]
The set $\cP_{0} \subset \cN(X) \otimes \C$ is open, and the restriction 
\[ \pi: \pi^{-1}(\cP_{0}(X)) \rightarrow \cP_{0}(X)\]
is a covering map. 
\end{Prop}

To show that $V(X)$ and $U(X)$ are connected, we need to introduce two more subsets:
\begin{equation*}
\cQ(X):=\{\Omega \in \cP(X) \mid (\Omega,\Omega)=0, (\Omega,\bar{\Omega})>0, r(\Omega)=1\}
\end{equation*}
and
\begin{equation*}
\cL(X):=\{\Omega \in \cK(X) \mid (\Omega,\delta) \not\in \R_{\leq 0} \ \ \forall \delta \in \Delta^{+}(X)\},
\end{equation*}
where $\Delta^{+}(X):=\{\delta \in \Delta(X) \mid r(\delta) >0\}$, and $r:\Halg \rightarrow H^0(X,\Z)\cong \Z$ is the first component.

\begin{Lem}\label{UandV}
We have 
\[V(X) = \{\sigma \in \Stab (X) \mid \cO_{x}\  \mathrm{is \  stable \  of\  phase\  } 1\  \forall x \in X\mathrm{,\  and\  } \pi(\sigma) \in \cQ(X)\}\]
and
\[ U(X)=\{\sigma \in \Stab(X) \mid \cO_{x} \  \mathrm{is\  stable\  } \forall x \in X\mathrm{,\  and\  } \pi(\sigma) \in \cP^{+}_{0}(X)\} \]
\end{Lem}

\begin{proof}
Without the condition $\pi(\sigma) \in \cP_0^+(X)$ this follows from the proof of \cite[Proposition $10.3$]{bridgeland_stability_2006}. In the third step of the proof, Bridgeland shows, using existence of slope stable spherical sheaves, that if all the skyscraper sheaves are $\sigma$-stable, then the central charge is in $\cP^+_0(X)$. Since we explicitly ask for the central charge to be in $\cP^+_0(X)$, we can skip this step. The rest of the proof works verbatim and gives the Lemma. 
\end{proof}

\begin{Lem}\label{lem:covering}
The restriction $\pi|_{V(X)}: V(X) \rightarrow \cQ(X)$ has open image and it is an homeomorphism onto its image. 
\end{Lem}

\begin{proof}
Since every stability condition in $V(X)$ is obtained by tilting, the map $\pi$ is injective when restricted to $V(X)$. So it is enough to show that it remains a covering on $V(X)$. Notice that Lemma \ref{UandV} and \cite[Proposition $9.4$]{bridgeland_stability_2006} imply that $U(X) \subset \Stab(X)$ is an open subset. Since $\pi(U(X)) \subset \cP_0(X)$ by definition and $U(X)$ is open, the map $\pi$ restricted to $U(X)$ is a covering onto its image. Moreover, a stability condition $\sigma \in U(X)$ is determined by its central charge $\pi(\sigma)$ up to even shifts, because the even shifts are the only elements of $\widetilde{GL_2^+}(\R)$ that fix the central charge. Let $A$ be a small neighborhood of $\pi(\sigma)$, since $\pi$ is a cover on $U(X)$, the inverse image $\pi^{-1}(A) \cap U(X)$ is homeomorphic to $A \times \Z$, where the second factors records the shift. Restricting to $\cQ(X) \cap A$ we see that $\pi^{-1}(A \cap \cQ(X)) \cap V(X)$ is contained in one component of $\pi^{-1}(A) \cap U(X)$, so $\pi|_{V(X)}$ induces an homeomorphism onto its image. 
\end{proof}

It is easy to see that the pairing $(-,\delta)|_{\cK(X)}: \cK(X) \rightarrow \C$ with any class $\delta$ is submersive when restricted to $\cK(X)$. In particular the preimage of a real half-line is a locally closed submanifold of real codimension one. These submanifolds are contained in real hyperplanes of $\Halg \otimes \C$. Bridgeland shows that the union of these hyperplanes is locally finite, and it uses this to show the following.

\begin{Lem}[{{\cite[Lemma $11.1$]{bridgeland_stability_2006}}}]\label{Lcontractible}
The set $\cL(X) \subset \cQ(X)$ is open and contractible. 
\end{Lem}

The proof of the following Proposition is the same as \cite[Proposition $11.2$]{bridgeland_stability_2006} with an extra step, but we reproduce the entire proof for readability. 

\begin{Prop}
The spaces $V(X)$ and $U(X)$ are connected.
\end{Prop}
\begin{proof}
First we claim that
\[ \cL(X) \subseteq \pi(V(X)). \] 
Lemma \ref{lem:covering} implies that $\pi(V(X))$ is an open subset of $Q(X)$ and $\pi|_{V(X)}:V(X) \rightarrow \pi(V(X))$ is an homemorphism. From Lemma \ref{Lcontractible} we get that $\cL(X) \cap \pi(V(X))$ is open in $\cL(X)$. Since $\cL(X)$ is connected, we only need to show that the intersection is closed in $\cL(X)$. 

Let $\Omega \in \overline{\cL(X) \cap \pi(V(X))} \subset \cL(X)$. Since 
$\pi$ is an homeomorphism restricted to $V(X)$, 
there exists a stability condition $\sigma \in \overline{V(X)}$ such that $\pi(\sigma)=\Omega$. If $\sigma$ is not in $V(X)$, Lemma \ref{UandV} implies that there is a strictly semistable skyscraper sheaf $\cO_x$; consider its Jordan-Holder factors $A_i$. From the definition of the category $\Coh^{\omega,B}(X)$, and the fact that $\omega$ is ample, it follows that if $\Im Z(A_i)=0$ and $r(A_i)=0$, then $\vv(A_i)$ is a multiple of $\vv(\cO_x)$. This implies that there is a Jordan-Holder factor $A$ with positive rank. Since $A$ has the same phase as $\cO_x$ we have $Z(A) \in \R_{<0}$, and we claim that $A$ is spherical. Recall that 
\[ \Re Z(A)=\frac{1}{2r}((\Delta^2 -2rs)+r^2\omega^2-(\Delta-rB)^2),\]
where $\vv(A)=(r,\Delta,s)$.
We have $\Im Z(A)=(\Delta-rB).\omega=0$, which implies $(\Delta-rB)^2 \leq 0$  by the Hodge index Theorem. Hence $\Delta^2 -2rs <0$, which means $A$ spherical, because it is stable. We found a spherical class $\delta:=\vv(A) \in \Delta^+(X)$ such that $(\Omega,\delta) < 0$, which contradicts the assumption that $\Omega \in \cL(X)$. 

To finish the proof, it is enough to show that $V(X)$ is connected. Since $\pi$ is a homeomorphism restricted to $V(X)$, we need to show that $\pi(V(X))$ is connected. Since $\cL(X)$ is connected, it is enough to show that $\cL(X)$ is dense in $\pi(V(X))$. If we assume non-emptiness of moduli stable of slope stable spherical sheaves, we have equality, as showed in \cite[Proposition $11.2$]{bridgeland_stability_2006}. If we do not have the non-emptiness, it could happen that for an $\Omega \in \pi(V(X))$ we have $(\Omega,\vv(E)) \not \in \R_{\leq 0}$ for every spherical torsion-free sheaf $E$, but $(\Omega, \delta) \in \R_{\leq 0}$ for some spherical class $\delta$ for which there are no corresponding sheaves. That is, the difference $\pi(V(X))-\cL(X)$ is contained in a locally finite union of locally closed submanifolds of real codimension one. Hence $\cL(X)$ is dense in $\pi(V(X))$ and $V(X)$ is connected because $\cL(X)$ is.
\end{proof}

\begin{Defi}
Define the distinguished component $\Stab^{\dagger}(X)$ as the connected component of the preimage $\pi^{-1}(\cP^+_0(X)) \subset \Stab(X)$ containing $U(X)$.
\end{Defi}

\begin{Rem}
As mentioned above, our definition differs from Bridgeland's original definition in \cite{bridgeland_stability_2006}. In ibidem it is defined as the connected component of $\Stab(X)$ containing the geometric chamber $U(X)$, and it is a Theorem that it gets mapped onto $\cP_{0}^+(X)$ via $\pi$. The proof requires existence of slope stable sheaves. In any case, our definition is sufficient to prove existence of slope stable shaves, because $\Stab^{\dagger}(X)$ contains the Gieseker chamber. So, once we prove the Main Theorem with our definition, it will also follow the Main Theorem for the standard definition of $\Stab^{\dagger}(X)$. 
\end{Rem}

\subsection{Equivalences preserving $\Stab^{\dagger}(X)$.}\label{subsec:preserveStab}
To conclude this section we want to show that there are enough equivalences between derived categories of K3 surfaces preserving $\Stab^{\dagger}(X)$. Every result here has already been shown by Hartmann in the appendix of \cite{hartmann_cusps_2012}. The idea is simple: since $\Stab^{\dagger}(X)$ is connected and contains the geometric chamber $U(X)$ it is enough to find a point $\sigma \in  \Stab^{\dagger}(X)$ that goes to the geometric chamber. This is easy to check thanks to the explicit description of $U(X)$ in Lemma \ref{UandV}. In our argument, there is the extra check that the equivalences preserve the domain $\cP_0^+(X)$, which (in the generality needed in this paper) is due to Huybrechts and Stellari \cite{Huybrechts_2005}. 

We start by recalling some generalities about Fourier-Mukai equivalences between K3 surfaces. Let $H$ be a polarization on $X$, let $\ww \in \Halg$ be a Mukai vector, and consider the moduli space $M:=M_H(\ww)$ of Gieseker stable sheaves of class $\ww$. Assume that it is a smooth projective surface, and that it is fine, i.e. it has a universal family $\cE \in \Coh(X \times M)$. We can consider the Fourier-Mukai transform with kernel the universal family:
\begin{align*}
    \Phi_{\cE}:D^b(M) &\rightarrow D^b(X) \\
    F &\mapsto q_*(\cE \otimes p^*F),
\end{align*}
where we denoted by $p,q$ the projections from $X \times M$ on the first and second factor, and where every functor is derived. 

\begin{Prop}[{{\cite{Muk:K3},\cite[Proposition $10.25$]{Huybrechts_FM}}}]\label{prop:FMequivalence}
If $M$ a smooth projective surface, and it is a fine moduli space, then the functor $\Phi_{\cE}$ is an equivalence. 
\end{Prop}

Recall that if $X$ is a K3 surface the structure sheaf $\cO_X$ is a spherical object. In particular we can consider the spherical twist $\ST_{\cO_X}$ around $\cO_X$ defined as follows.

\begin{Defi}
Let $S \in D^b(X)$ be a spherical object. The \textit{spherical twist} around $S$, denoted by $\ST_S(-)$ is defined, for every $E \in D^b(X)$, as the cone of the evaluation map:
\[ \RHom(S,E) \otimes S \rightarrow E \rightarrow \ST_S(E) \rightarrow \RHom(S,E) \otimes S[1]. \]
\end{Defi} 

\begin{Prop}[{{\cite[Proposition $2.10$]{seidel_2001},\cite[Proposition $8.6$]{Huybrechts_FM}}}]
The spherical twist $\ST_S$ around a spherical object $S$ is an equivalence.
\end{Prop}

Every Fourier-Mukai equivalence $\Phi_{\cE}: D^b(X) \xrightarrow{\sim} D^b(X')$ between derived categories of K3 surfaces induces a map in integral cohomology. It is the cohomological Fourier-Mukai transform with kernel the Mukai vector $v(\cP)$:
\[ \Phi_{\cE}^H(x)=q_*(v(\cP).p^*(x)), \]
where $p$ and $q$ are the projections $X \times X'$ on the first and second factor. It is well known that it is a Hodge isometry, if we equip the integral cohomology with the following weight two Hodge structure :
\[ H^*(X,\C)=H^{2,0}(X) \oplus (H^0(X,\C) \oplus H^{1,1}(X) \oplus H^{4}(X)) \oplus H^{0,2}(X).\]
In particular it induces an isometry $\Phi_{\cE}^H:\Halg \xrightarrow{\sim} H^*_{\textrm{alg}}(X',\Z)$ between the algebraic parts. Every derived equivalence $\Phi: D^b(X) \xrightarrow{\sim} D^b(X')$ is of Fourier-Mukai type, and the induced isometry does not depend on the kernel. If the kernel is not given, we will denote the induced isometry by $\Phi^H$. 

The equivalences we are interested in are the following:
\begin{enumerate}
    \item Tensor product by a line bundle: $E \mapsto E \otimes L$;
    \item Shift: $E \mapsto E[1]$
    \item The spherical twist $\ST_{\cO_X}$ around $\cO_X$;
    \item Fourier-Mukai transforms $\Phi_{\cE}$ associated to a fine two dimensional moduli space of Gieseker stable sheaves. 
\end{enumerate}

\begin{Prop}[{{\cite[Remark $5.4$ and Proposition $5.5$]{Huybrechts_2005}}}]\label{prop:orientation}
The isometries induced by the equivalences of type $(1)-(4)$ preserve the set $\cP_0^+(X)$. 
\end{Prop}

An equivalence $\Phi: D^{b}(X) \xrightarrow{\sim} D^{b}(X')$ induces an isomorphism of spaces of stability conditions
\begin{align*}
    \Phi_{*}:\Stab(X) &\xrightarrow{\sim} \Stab(X')\\
    (Z,\cP) &\mapsto (Z\circ \Phi^{-1},\cP')
\end{align*}
where $\cP'(t)=\Phi(\cP(t))$. We say that $\Phi$ preserves the distinguished component if \[\Phi_*(\Stab^{\dagger}(X))=\Stab^{\dagger}(X').\]

\begin{Lem}\label{preservation_condition}
Let $\Phi: D^{b}(X) \xrightarrow{\sim} D^{b}(X')$ be a derived equivalence of K3 surfaces of type $(1)-(4)$. Assume that there exists $\sigma'=(Z',\cP') \in \Stab^{\dagger}(X')$ such that the objects $\Phi(\cO_{x})$ are $\sigma'$-stable and such that $\Omega_{Z'} \in \cP_{0}^+(X').$ Then $\Phi$ preserves the distinguished component.
\end{Lem}

\begin{proof}
It is enough to show that a point of $\Stab^{\dagger}(X)$ gets mapped to $\Stab^{\dagger}(X')$. Consider $(Z,\cP)=\sigma:=\Phi_*^{-1}(\sigma')$. By assumption the skyscraper sheaves $\cO_{x}$ are all $\sigma$-stable. Proposition \ref{prop:orientation} shows that the induced isometry in cohomology sends $\cP^{+}_{0}(X)$ to $\cP^+_0(X')$. In particular $\Omega_{Z}=(\Phi^H)^{-1}(\Omega_{Z'})$ is in $\cP_{0}^{+}(X)$. Then, Lemma \ref{UandV} implies $\sigma \in U(X) \subset \Stab^{\dagger}(X)$.
\end{proof}

In order to show that the equivalences we are interested in preserve the distinguished component, we need a standard result about the large volume limit. Let $H \in \NS(X)$ be an ample class, and $B \in \NS_{\Q}(X)$ a rational class. Consider the stability condition $\sigma_{\alpha H,B}$.

\begin{Thm}[{{\cite[Proposition $14.1$]{bridgeland_stability_2006}}} and {{\cite[Section 6]{Toda_Moduli}}}]\label{large_volume}
Let $\vv=(r,\Delta,s)$ be a primitive Mukai vector, with either $r >0 $ or $r=0$ and $\Delta \neq 0$ effective. Then there exists an $\alpha_{0}$ such that, for every $\alpha \geq \alpha_{0}$, an object $E \in D^{b}(X)$ of class $\vv$ is $\sigma_{\alpha H,B}$-stable if and only if it is a shift of a $B$-twisted $H$-Gieseker stable sheaf. 
\end{Thm}

\begin{Cor}[{{\cite[Lemma $7.2$, Propositions $7.5$ and $7.6$]{hartmann_cusps_2012}}}]\label{preserve}
The equivalences of type $(1)-(4)$ preserve the distinguished component.
\end{Cor}

\begin{proof}
The equivalences of type $(1)$ and $(2)$ send skyscraper sheaves to (shifts of) skyscraper sheaves, so by Lemma \ref{UandV} and Proposition \ref{prop:orientation} they preserve the geometric chamber $U(X)$ and, a fortiori, the distinguished component.

For the remaining two $(3)$ and $(4)$ we use Lemma \ref{preservation_condition}: it is enough to find a stability condition $\sigma \in \Stab^{\dagger}(X)$ such that $\Phi(\cO_x)$ are $\sigma$-stable, and whose central charge satisfies $\Omega_Z \in \cP_0^+(X)$. 
For the spherical twist, notice that $\ST_{\cO_X}(\cO_x)=\mathrm{m}_x$, the ideal sheaf of the point $x$. These are Gieseker stable, so by choosing $\sigma_{\alpha H,B}$ appropriately as in Theorem \ref{large_volume} we find a $\sigma \in V(X)$ that works. Similarly, if $\cE$ is a universal family over a Gieseker moduli space, the objects $\Phi(\cE)(\cO_x)$ are Gieseker stable, and again we conclude by Theorem \ref{large_volume}.
\end{proof}

\section{Review: Hyperk\"{a}hler varieties and Moduli spaces} 
\label{sec:geometry}

In this section we give a short review on hyperk\"{a}hler varieties, and basic facts about moduli spaces of stable sheaves and stable complexes.

\begin{Defi}
A projective hyperk\"{a}hler variety is a smooth projective complex variety, which is simply connected and such that $H^0(X,\Omega_X^2)$ is one dimensional and spanned by a symplectic $2$-form. 
\end{Defi}

On the $H^2(X,\Z)$ there is a natural integral quadratic form $q_X$, called Fujiki-Beauville-Bogomolov form. It is a deformation invariant, and has signature $(3,b_2(X)-3)$. It satisfies the Fujiki relation
\begin{equation*}
    \int_X{\alpha^n}=\lambda_Xq_X(\alpha)^n, \qquad \alpha \in H^2(X,\Z).
\end{equation*}
The constant $\lambda_X$ is called Fujiki constant, and it is  deformation invariant. 

Fix $\sigma=(Z,\cP) \in \Stab(X)$ a stability condition, a phase $\phi \in \R$ and a Mukai vector $\vv \in \Halg$. Consider the moduli stack $\mathfrak{M}_{\sigma}(\vv,\phi)$ of $\sigma$-semistable objects of class $\vv$ and phase $\phi$. Its objects over $S$ are $S$-perfect complexes $\cE \in D^b_{S-\textrm{perf}}(S \times X)$, whose restriction over a closed point $s \in S$ belongs to $\cP(\phi)$ and has class $\vv$.  The following is a collection of result by Toda \cite{Toda_Moduli}, Inaba \cite{inaba_smoothness_2010}, and Lieblich \cite{Lieblich_complexes}. 

\begin{Thm}\label{thm:symp}
Let $X$ be a K3 surface, $\vv \in \Halg$ and $\sigma \in \Stab^{\dagger}(X)$. Then $\fM_{\sigma}(\vv,\phi)$ is an Artin stack of finite type over $\C$. Denote by $\fM_{\sigma}^s(\vv,\phi) \subseteq \fM_{\sigma}(\vv,\phi)$ the open substack parametrizing $\sigma$-stable objects. If $\fM_{\sigma}^s(\vv,\phi) = \fM_{\sigma}(\vv,\phi)$, then $\fM_{\sigma}(\vv,\phi)$ is a $\mathbb{G}_m$-gerbe over its coarse moduli space $M_{\sigma}(\vv,\phi)$, which is a smooth, proper, symplectic algebraic space with expected dimension $\vv^2+2$.   
\end{Thm}

In particular the assumptions are satisfied if $\vv$ is primitive and $\sigma$ is $\vv$-generic. The phase $\phi$ is determined by the rest of the data, up to an even integer. Since the corresponding moduli spaces are isomorphic via shifts, from now on we drop the $\phi$ from the notation, and denote a moduli space simply by $M_{\sigma}(\vv)$.

The moduli spaces $M_{\sigma}(\vv)$ are not necessarily fine, but if $\vv$ is primitive and $\sigma \in \Stab^{\dagger}(X)$ is $\vv$-generic, they are equipped with a quasi-universal family unique up to equivalence, by \cite[Theorem A$.5$]{mukai_symplectic_1984}.

\begin{Defi}\label{quasi-universal}
Let $M=M_{\sigma}(\vv,\phi)$ be a coarse moduli space. 
\begin{enumerate}
    \item A flat family $\cE$ on $M \times X$ is called a quasi-family of objects in $\fM_{\sigma}(\vv)$ if, for all closed points $m \in M$, there exists an integer $\rho >0$, and an element $E \in \fM_{\sigma}(\vv,\phi)(\C)$ such that $\cE|_{{t}\times X} \cong E^{\oplus \rho}$. If $M$ is connected $\rho$ is independent of $m$, and is called the similitude of $\cE$.
    \item Two quasi-families are equivalent if there exists vector bundles $V$ and $V'$ on $M$ such that $\cE' \otimes p_M^*V \cong \cE \otimes p_M^*V'$.
    \item a quasi-family $\cE$ is called quasi-universal if, for every scheme $T$ and for any quasi-family $\cT$ on $M \times X$, there exists a unique morphism $f: M \rightarrow T $ such that $f^*\cE$ and $\cT$ are equivalent.
\end{enumerate} 
\end{Defi}

Projectivity of the coarse moduli space $M_{\sigma}(\vv)$ was proved in \cite{bayer_projectivity_2013}. The problem is that in general for moduli spaces of stable complexes there is no obvious GIT construction. Bayer and Macr\`i constructed a divisor class $l_{\sigma} \in \NS(M_{\sigma}(\vv))_{\R}$ as follows: 
\[ C \mapsto l_{\sigma}.C:= \Im (-\frac{Z(v(\Phi_{\cE}(\cO_C)))}{Z(\vv)}), \]
where $C \in M_{\sigma}(\vv)$ is a curve, and $\cE$ is a quasi-universal family.

\begin{Thm}[{{\cite[Theorem $4.1$ and Remark $4.6$]{bayer_projectivity_2013}}}]\label{nefdivisor}
Let $\vv$ be a primitive Mukai vector, $\sigma \in \Stab^{\dagger}(X)$ a $\vv$-generic stability condition.  Then the class $l_{\sigma}$ defined above is ample. 
\end{Thm}

One of the key steps in the proof is the use of \cite[Lemma $7.3$]{bayer_projectivity_2013} to reduce to to the classical case of Gieseker stable sheaves. The same reduction argument also shows  irreducibility of $M_{\sigma}(\vv)$. The following statement summarizes the discussion above. 

\begin{Cor}\label{Cor:projectivity}
Let $X$ be a K3 surface, $\vv \in \Halg$ a primitive vector with $\vv^2 \geq -2$. Let $\sigma \in \Stab^{\dagger}(X)$ be a $\vv$-generic stability condition. Then if $M_{\sigma}(\vv)$ is non-empty, it is a smooth, projective symplectic variety of dimension $\vv^2+2$ and it consists of stable objects. 
\end{Cor}

\begin{Rem}
There is a subtlety here. To use the arguments in \cite[Lemma $7.3$]{bayer_projectivity_2013} we need to know the Main Theorem for the case of a primitive vector $\vv \in \Halg$ with $\vv^2=0$. This is not a problem for us, because we will not use projectivity in the arguments at all, and in fact we will also reprove irreducibility via a deformation argument. Moreover, the Main Theorem for the square zero case will be proved separately from the positive square case, so we will be able to apply the above corollary to get projectivity in the latter case.
\end{Rem}

\begin{Defi}\label{Def:MukaiHom}
Let $\vv \in \Halg$ a primitive class with $\vv^2>0$, and let $\sigma \in \Stab^{\dagger}(X)$ be a $\vv$-generic stability condition, in particular every $\sigma$-semistable object is $\sigma$-stable. We define the Mukai homomorphism $\theta_{\vv}:\vv^{\perp} \rightarrow H^{2}(M_{\sigma}(\vv),\Z)$ by 
\begin{equation}\label{eq:mukai_hom}
    \theta_{\vv}(\xx)=\frac{1}{\rho}[\Phi^{H}_{\cE}(\xx^{\vee})]_1
\end{equation} 
where $\cE$ is a quasi-universal family of similitude $\rho$, and $[-]_1$ is the component belonging to $H^{2}(M_{\sigma}(\vv),\Z)$. If $\vv^2=0$, the same formula gives a well defined map $\theta_{\vv}:\vv^{\perp}/\Z.\vv \rightarrow H^{2}(M_{\sigma}(\vv),\Z)$
\end{Defi}

It can be shown that it does not depend on the quasi-universal family $\cE$ if we restrict to $\vv^{\perp}$.

\begin{Rem}
The definition of Mukai homomorphism in \cite{ogrady_weight-two_1995} and \cite{Yoshioka_main} is 
\[ \frac{1}{\rho}[p_{M_{\sigma}(\vv)*}\ch(\cE)p_X^*(\sqrt{\td_X}x^{\vee})]_1. \]
This is equivalent to ours. Indeed, recall the definition of the Mukai vector \[v(\cE)=\ch(\cE)p^*_{M_{\sigma}(\vv)}\left(\sqrt{\td_{M_{\sigma}(\vv)}}\right)p^*_{X}\left(\sqrt{\td_X}\right).\] The degree two component of Todd class $\td_{M_{\sigma}(\vv)}$ is $0$, because $M_{\sigma}(\vv)$ has trivial canonical bundle, so its square root does not contribute to the degree two component. 
\end{Rem}

Recall that an anti-equivalence is an equivalence from the opposite category $D^b(X)^{\textrm{op}}$ to $D^b(X')$. Every anti-equivalence is given by a composition of an equivalence and the dualizing functor $\RcHom(-,\cO_X')$.

\begin{Prop}\label{prop:equivalences}
Let $X,X'$ be two K3 surfaces, $\vv,\vv'$ two Mukai vectors on $X$ and $X'$ respectively. Let $\sigma \in \Stab^{\dagger}(X)$ be a $\vv$-generic stability condition on $X$, and $\sigma' \in \Stab^{\dagger}(X')$ a $\vv'$-generic stability condition on $X'$. Assume that there is an (anti)-equivalence $\Phi: D^{b}(X) \xrightarrow{\sim} D^{b}(X')$ that induces an isomorphism $M_{X,\sigma}(\vv) \xrightarrow{\sim} M_{X',\sigma'}(\vv')$. If $\vv^2>0$ we have a commutative diagram
\begin{center}
\begin{tikzcd}
\vv^{\perp} \arrow[d, "\theta_{\vv}"'] \arrow[r, "\Phi^H"] & \vv'^{\perp} \arrow[d, "\theta_{\vv'}"'] \\
{H^{2}(M_{X,\sigma}(\vv),\Z)} \ar[equal]{r}        & {H^2(M_{X',\sigma'}(\vv'),\Z)}          
\end{tikzcd}
\end{center}
The analogous statement holds if $\vv^2=0$. In particular, if $\theta_{\vv}$ is a Hodge isometry then so it is $\theta_{\vv'}$.
\end{Prop}

\begin{proof}
This follows from the same computations of \cite[Propositions $2.4$ and $2.5$]{Yoshioka_main}.
\end{proof}

We conclude this section with two concrete examples. These will be the main geometric input in the proof of the Main Theorem: the Hilbert scheme is the base case to which we want to reduce, and moduli spaces of vector bundles on the fibers of an elliptic K3 surface will be Fourier-Mukai partners. 

\begin{Ex}\label{ex:hilbert_scheme}
Let $X$ be a K3 surface, consider the vector $\vv=(1,0,1-n)$ with $n\ \geq 2$. A sheaf with class $\vv$ has trivial double dual, and the natural map $\cF \rightarrow \cF^{\vee \vee} \cong \cO_X $ has cokernel of length $n$. Every such sheaf is torsion free with rank one, so it is Gieseker stable with respect to any polarization $H$. The natural map 
\begin{align*}
M_H(\vv) &\rightarrow \Hilb^n(X)\\
\cF &\mapsto (\cF \onto \cF^{\vee \vee}/\cF)
\end{align*}
is an isomorphism, where $\Hilb^n(X)$ is the Hilbert scheme of $n$ points. The Mukai homomorphism is compatible with such identification, and for $\Hilb^n(X)$ it is an isometry \cite[Section $6$]{beauville_1983}.
\end{Ex}

\begin{Ex}\label{ex:Fm_partner}
Let $X$ be an elliptic K3 surface, assume that $\Pic X=\Z s \oplus \Z f$, where $f$ is the class of a fiber, and $s$ is the class of a section. The intersection form with respect to this basis is
\[
\left(
\begin{array}{cc}
    -2 & 1  \\
     1 & 0
\end{array}
\right),
\]
so the Picard group is an hyperbolic plane. Consider a Mukai vector 
\[
\ww=(0,\alpha f,\beta), \textrm{ with } \alpha>0, \beta \neq 0 \textrm{ and } \gcd (\alpha,\beta)=1.
\]
Let $H$ be a generic polarization, i.e. one for which all the $H$-semistable sheaves are stable, they exist because we are assuming $\beta \neq 0$. It follows from the definition and the Grothendieck-Riemann-Roch Theorem that if $E$ is a slope-stable bundle of rank $\alpha$ and degree $\beta$ supported on a smooth fiber $C \in |f|$, then it is $H$-stable as a torsion sheaf on $X$. So the Gieseker moduli space $M:=M_H(\ww)$ is a smooth, projective, symplectic surface.
\end{Ex}

\section{Wall-crossing: Semirigid case}
\label{sec:Wall crossing}

The objective of this section is to show that the statement of the Main Theorem is preserved under wall-crossing, when the Mukai vector $\vv$ is spherical ($\vv^2=-2$) or isotropic ($\vv^2=0$). The precise setup is the following. We fix a K3 surface $X$, a primitive Mukai vector with $\vv^2=-2$ or $\vv^2=0$, and $\cW \subset \Stab^{\dagger}(X)$ a wall for $\vv$. We denote the adjacent chambers with $\cC_+$ and $\cC_-$, we also denote with $\sigma_{\pm}$ a generic stability condition in $\cC_{\pm}$, and with $\sigma_0=(Z_0,\cP_0)$ a generic stability condition on the wall. The following is the main result of this section. 

\begin{Thm}
\label{semirigid1}
Let $X$ be a K3 surface, $\vv$ be a primitive vector, with $\vv^{2}=-2$ or $\vv^2=0$. Let $\cW$ be a wall for the wall and chamber decomposition for $\vv$.
\begin{enumerate}
    \item If $\vv^{2}=-2$, then $M_{\sigma_{+}}(\vv) \neq \emptyset $ implies $M_{\sigma_{-}}(\vv) \neq \emptyset$. 
    \item If $\vv^{2}=0$, then there exists a spherical, $\sigma_{0}$-stable object $S$ such that either $\ST_{S}$ or $\ST_{S}^{\pm 2}$ induce an isomorphism $M_{\sigma_{+}} \xrightarrow{\sim} M_{\sigma_{-}}$.
\end{enumerate}
\end{Thm}

This immediately implies the invariance of the Main Theorem under wall-crossing for spherical and isotropic classes, see Corollary \ref{yoshiokaez}. 

\subsection{Lattice associated to the wall}
The key tool to study wall-crossing is a rank two lattice $\cH$ associated to our setup. It was introduced in \cite[Section $5$]{mmp}, for the case of a vector with $\vv^2>0$. In that case, $\cH$ is always hyperbolic, while if $\vv^2 \leq 0$ it can also be negative semi-definite.   

\begin{Defi}\label{def:lattice}
Define the lattice associated to $\cW$ as 
\[
\cH:=\{\ww \in \Halg \mid \Im \frac{Z_0(\ww)}{Z_0(\vv)}=0 \}.
\]
\end{Defi}

\begin{Prop}[{{\cite[Proposition $5.1$]{mmp}}}]
The lattice $\cH$ has the following properties. 
\begin{enumerate}
\item It is a rank $2$ primitive sublattice of $\Halg$.
\item For every $\sigma_{+}$-stable object $E$ of class $\vv$, the Mukai vectors of its Harder-Narasimhan factors with respect to $\sigma_{-}$ are contained in $\cH$. 
\item If $E$ is $\sigma_{0}$-semistable of class $\vv$, then the Mukai vectors of its Harder-Narasimhan factors with respect to $\sigma_{-}$ are contained in $\cH$.
\item If $E$ is $\sigma_{0}$-semistable of class $\vv(E) \in \cH$, then its Jordan-Holder factors have Mukai vector in $\cH$. 
\end{enumerate}
\end{Prop}

\begin{Lem}\label{signature}
The lattice $\cH$ is either hyperbolic or negative semi-definite
\end{Lem}
\begin{proof}
Acting with $\widetilde{GL_2(\R)}$ we can assume $\sigma_0$ be such that $Z_0(\vv)=-1$. Write $Z_0=(-,\Omega)$, since $\sigma_0 \in \Stab^{\dagger}(X)$ we have $\Omega \in \cP(X)$, in particular $(\Im\Omega)^{2}>0$. By definition, the lattice $\cH$ is contained in the orthogonal complement to $\Im \Omega_{Z}$. The Mukai lattice has signature $(2,\rho(X))$, hence the orthogonal to $\Im \Omega$ has signature $(1,\rho(X))$. This implies that $\cH$ cointains classes with negative square, hence the thesis. 
\end{proof}

\begin{Rem}
Notice that if $\cH$ were negative definite, there would be at most two spherical classes up to sign, and no isotropic class. It is easy to see that in this case every spherical object with class in $\cH$ remains stable on the wall. 
\end{Rem}

We are going to need a couple of technical lemmas, that we recall here. 

\begin{Lem}[Mukai's Lemma, {{\cite[Lemma $5.2$]{bridgeland_stability_2006}}}]\label{Mukailemma}
Let $0 \rightarrow A \rightarrow E \rightarrow B \rightarrow 0$ be a short exact sequence inside a heart $\cA \subset D^{b}(X)$. If $\Hom(A,B)=0$, then 
\[ \ext^{1}(E,E) \geq \ext^{1}(A,A) + \ext^{1}(B,B). \]
\end{Lem}

\begin{Lem}\label{simpleext}
Let $\cW \subset \Stab(X)$ be a wall for $\vv$, $\sigma_0 \in \cW$ a generic stability condition, and $\sigma_+$ a stability condition on one of the adjacent chambers. Consider a short exact sequence in $\cA_{\sigma_+}$
\[ 0\rightarrow S \rightarrow E \rightarrow T \rightarrow 0, \]
where $S$ and $T$ are $\sigma_0$-stable of the same phase, and $v(E)=\vv$. Assume that $\phi_{\sigma_+}(S)<\phi_{\sigma_+}(E)<\phi_{\sigma_+}(T)$, and $\Hom(T,E)=0$, then $E$ is $\sigma_+$-stable. 
\end{Lem}
\begin{proof}
Assume $E$ is not $\sigma_+$-stable, and consider a stable destabilizing subobject $A \hookrightarrow E$ in $\cA_{\sigma_+}$. By assumption we have 
$\phi_{\sigma_+}(A) > \phi_{\sigma_+}(E) > \phi_{\sigma_+} (S).$
If $\phi_{\sigma_+}(A) \geq \phi_{\sigma_+}(T)$ we would get $\Hom(A,T)=0$ by stability. Then, the morphism $A \hookrightarrow E$ would factor via $S$, but $\Hom(A,S)=0$ by stability. So we have 
\[ \phi_{\sigma_+}(S) < \phi_{\sigma_+}(A) < \phi_{\sigma_+}(T). \]
This implies that $A$ is $\sigma_0$-semistable of the same $\sigma_0$-phase as $S$ and $T$. Since $S$ and $T$ are simple objects in the abelian category of $\sigma_0$-semistable objects of their phase (i.e. they do not have proper subobjects), we see that $A=S$ or $A=T$. The first case contradicts $\phi(A) > \phi(E)$ and the second one $\Hom(T,E)=0$.
\end{proof}

\subsection{Spherical Mukai vector}\label{subsec:spherical}
Here we prove part $(1)$ of Theorem \ref{semirigid1}. Fix a Mukai vector $\vv$ with $\vv^2=-2$. Given a $\sigma_+$-stable spherical object $E$ with $v(E)=\vv$, we want to construct a $\sigma_-$-stable spherical object $E'$ with the same Mukai vector. The idea is to deform the stability condition $\sigma_+$ to a generic stability condition $\sigma_0$ on the wall $\cW$ and take the Jordan-Holder filtration of $E$. It turns out (Proposition \ref{jhforspherical}) that $E$ has only two Jordan-Holder factors, although they can appear multiple times. Call this two Jordan-Holder factors $S$ and $T$, and their classes $\ss$ and $\tt$. They are $\sigma_0$-stable spherical objects, so they are $\sigma_-$-stable too, since the condition that an object is stable is open in $\Stab(X)$. To construct the desired object $E'$, we will construct inductively, starting from $S$ and $T$, a $\sigma_-$-stable spherical object with class $\vv'$ for every spherical $\vv'$ that is a linear combination of $\ss$ and $\tt$ with positive coefficients. Since $E$ has a Jordan-Holder filtration with factors $S$ and $T$, its class $\vv$ is of that form.

\begin{Prop}\label{jhforspherical}
Let $E \in M_{\sigma_{+}}(\vv)$. Assume that it is not stable on the wall. Then there are two $\sigma_{0}$-stable spherical objects that appear as Jordan-Holder factors of $E$, possibly with multiplicity.
\end{Prop}

\begin{proof}
Assume that $E$ gets destabilized. From Lemma \ref{Mukailemma} it follows that its Jordan-Holder factors with respect to $\sigma_{0}$ are all spherical. Since $\vv$ is primitive, it must have at least two different factors $S,T$, call their classes $\ss,\tt$. Since $S,T$ are stable and non isomorphic we have $(s,t)=\ext^{1}(S,T) \geq 0$. This in turn implies that $\ss$ and $\tt$ are linearly independent over $\R$. Indeed, if we could write $s= \lambda t$, then $\lambda$ would be positive, because $S$ and $T$ have the same $\sigma_0$-phase, hence $(\ss,\tt)=-2\lambda <0$. The argument to show that these are the only Jordan-Holder factors is different in the case when $\cH$ is semi-definite and in the case when it is hyperbolic.

{\textbf{Semi-definite case.}} 
From the linear independence it follows that $(\ss,\tt)=2$. The spherical classes of $\sigma_0$-stable objects of the same phase as $E$ lie on two parallel half-lines, as shown in Figure \ref{lattices}(A). Furthermore, the product of two classes is positive if and only if they lie on different lines. We conclude that, up to shifts, $S$ and $T$ are the only two $\sigma_{0}$-stable spherical objects with classes in $\cH$.

{\textbf{Hyperbolic case.}} 
In this case we have $m:=(\ss,\tt) \geq 3$. Then, by the following argument from \cite{mmp}, we see again that, up to shifts, $S$ and $T$ are the only two $\sigma_{0}$-stable spherical objects. Assume $\rr$ is the class of another $\sigma_{0}$-stable spherical object, we can write $\rr=x\ss+y\tt$. We see that
\begin{align*}
    (\ss,\rr) \geq 0 &\implies y \geq \frac{2x}{m} \\
    (\tt,\rr) \geq 0 & \implies y \leq \frac{mx}{2} \\
    (\rr,\rr)=-2 & \implies -2x^{2}+2mxy-2y^{2}=-2,
\end{align*}
which is easily seen to be contradictory. 
\end{proof}

\begin{Rem}\label{twospherical}
Assume that $E$ is $\sigma_+$-stable, spherical and not stable on the wall. Proposition \ref{jhforspherical} gives two spherical classes $\ss,\tt$. They are a basis for $\cH_{\R}$, and the class $\vv=v(E)$ is a linear combination of $\ss$ and $\tt$ with positive coefficients. Writing the quadratic form with respect to the basis $\{\ss,\tt\}$ we get:
\begin{align*}
   & -2x^2+4xy-2y^2 \quad \textrm{in the semi-definite case},\\
   & -2x^2+2mxy-2y^2 \textrm{, with } m:=(\ss,\tt) >2 \quad \textrm{in the hyperbolic case}.
\end{align*}
In both cases there are infinitely many spherical classes in the lattice $\langle \ss,\tt \rangle$ spanned by $\ss$ and $\tt$. In the hyperbolic case there are no isotropic classes, because $\sqrt{m^2-4}$ is irrational if $m\geq 3$. The spherical classes live on two branches of an hyperbola in the hyperbolic case, and on two parallel lines in the semi-definite case.

\begin{figure}[ht]
\begin{subfigure}{0.49\textwidth}
\includegraphics[width=0.99\linewidth]{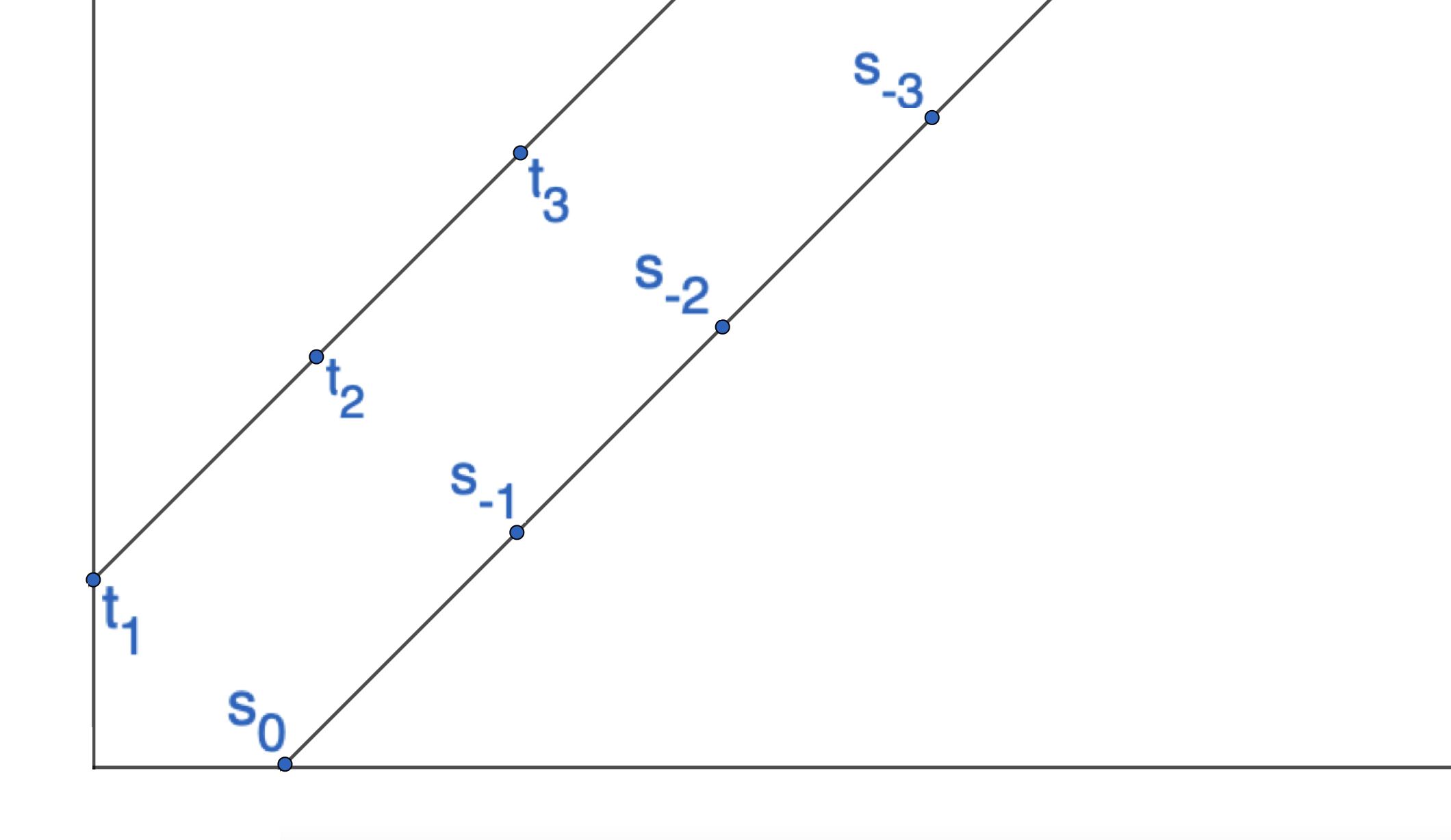}
\caption{Semi-definite case}
\label{fig:semidefLattice}
\end{subfigure}
\begin{subfigure}{0.49\textwidth}
\includegraphics[width=0.92\linewidth]{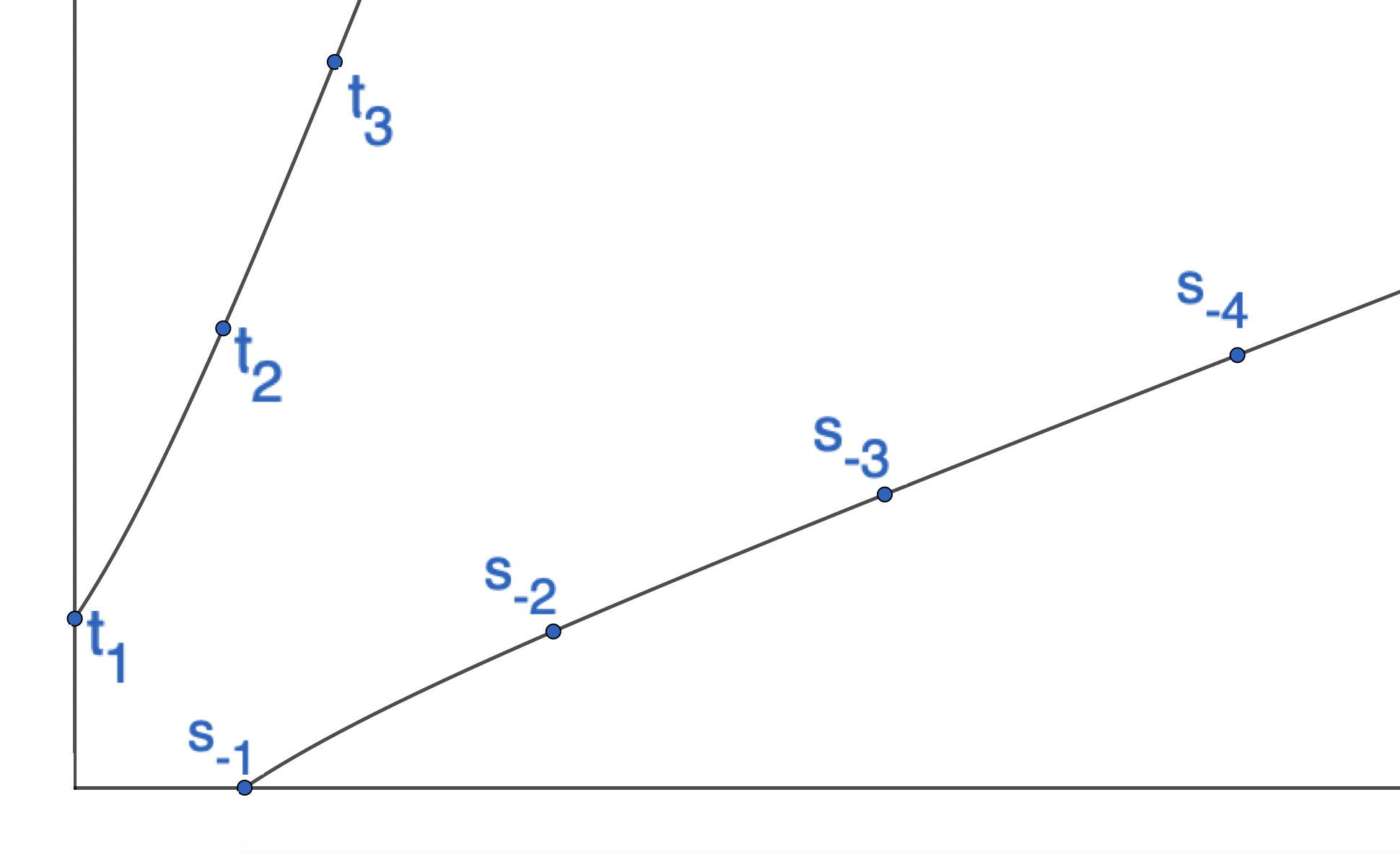}
\caption{Hyperbolic case}
\label{fig:hyperbolicLattice}
\end{subfigure}
\caption{Spherical classes in $\langle \ss,\tt \rangle$}
\label{lattices}
\end{figure}

Assume without loss of generality that $\phi_{\sigma_+}(\tt) > \phi_{\sigma_+}(\ss)$. Consider the spherical classes that are linear combination of $\ss$ and $\tt$ with positive coefficient, ordered with respect to $\sigma_+$ phase. Call $\tt_i$ for $i\geq 1$ the classes on the upper branch, and $\ss_i$ for $i \leq 0$ the classes on the lower branch, as shown in the Figure \ref{lattices}. They can also be defined inductively by 
\[ 
\begin{cases}
\tt_1=\tt\\
\tt_2=\rho_{\tt}(\ss),\\
\tt_{i+1}=-\rho_{\tt_i}(\tt_{i-1}).
\end{cases}
\begin{cases}
\ss_0=\ss\\
\ss_{-1}=\rho_{\ss}(\tt),\\
\ss_{-i-1}=-\rho_{\ss_{-i}}(\ss_{-i+1}).
\end{cases}
\]
where $\rho_{\ss}(\vv):=\vv+(\ss,\vv)\ss$. This is clear in the semi-definite case, because $\tt_i$ is the midpoint of the segment $\overline{\tt_{i-1},\tt_{i+1}}$, and is also easy to see in the hyperbolic case by writing down the previous reflections in coordinates with respect to $\ss$ and $\tt$. 
\end{Rem}

With this we are ready to show the first part of Theorem \ref{semirigid1}.

\begin{proof}[Proof of Theorem \ref{semirigid1}(1)]
Let $E \in M_{\sigma_+}(\vv)$, we want to show that there exists a $\sigma_-$-stable object with Mukai vector $\vv$. Let $\phi$ be the phase of $E$ with respect to $\sigma_0$; we can assume up to shifts that $0<\phi \leq 1$. Assume $E$ is not stable on the wall, otherwise we are done. From Proposition \ref{jhforspherical} and Remark \ref{twospherical} we get that $\vv=\tt_i$ or $\vv=\ss_{-i}$ for some $i$. Assume $\vv=\ss_{-i}$, the other case is analogous. We prove existence of $\sigma_-$-stable objects of class $\ss_{-i}$ by induction on $i$. Lemma \ref{jhforspherical} implies that there is a $\sigma_0$-stable object $S$ of class $\ss_0=\ss$, and a $\sigma_0$-stable object $T$ of class $\tt_1$. Define $S^-_{-i}$ inductively as
\[ S^{-}_{-i-1} := 
\begin{cases}
\ST_{S}(T) & \text{if $i=0$} , \\
\ST_{S^{-}_{-i}}(S_{-i+1}^{-})[-1] & \text{if $i>0$}
\end{cases}
\]
By stability of $S$ and $T$ we have a short exact sequence in $\cP(\phi)$
\[
0 \rightarrow T \rightarrow \ST_{S}(T) \rightarrow \Ext^{1}(S,T) \otimes S \rightarrow 0.
\]
Since $S,T$ are simple in the abelian category $\cP(\phi)$ and $\phi_{-}(T)<\phi_-(S)$, we can apply Lemma \ref{simpleext} and conclude that $S^-_{-1}$ is $\sigma_{-}$-stable. Furthermore, if we take $\sigma_-$ close to the wall, $S$ and $T$ lie in the heart $\cA_{\sigma_-}=\cP_{\sigma_-}(0,1]$, and so does $S^-_{-1}$.

Now, assume by induction that $S^-_{-j}$ is $\sigma_-$-stable for every $j \leq i$, and that it lies in the heart $\cA_{\sigma_-}$. We want to 
show that the same holds for $S^-_{-i-1}$. First we claim that $\RHom(S^{-}_{-i},S^{-}_{-i+1})$ is concentrated in degree zero.
Indeed $S^{-}_{-i},S^{-}_{-i+1}$ are two $\sigma_{-}$-stable objects with $\phi_{\sigma_{-}}(S^{-}_{-i})<\phi_{\sigma_{-}}(S^{-}_{-i+1})$. Therefore $\Hom^{2}(S^{-}_{-i},S^{-}_{-i+1})$ vanishes by stability and Serre duality. From the inductive definition and Serre duality we get
\begin{align*}
\ext^{1}(S^{-}_{-i},S^{-}_{-i+1}) &= \ext^{1}(S^{-}_{-i+1},S^{-}_{-i}) \\
&= \hom(S^{-}_{-i+1},\ST_{S^{-}_{-i+1}}(S^{-}_{-i+2})) \\
&= \hom(\ST_{S^{-}_{-i+1}}^{-1}(S^{-}_{-i+1}),S^{-}_{-i+2})\\
&=\hom(S^{-}_{-i+1}[1],S^{-}_{-i+2})
\end{align*}
which is zero because it is a negative Ext between two objects of a heart. 

This shows that we have the exact triangle
\begin{equation}\label{eq:inductiveTwist}
    S^{-}_{-i-1} \rightarrow \Hom(S^{-}_{-i},S^{-}_{-i+1}) \otimes S^{-}_{-i} \rightarrow S^{-}_{-i+1} \rightarrow S^{-}_{-i-1}[1] 
\end{equation}
Taking the long exact sequence of cohomology with value in the heart $\cA_{\sigma_-}=\cP_{\sigma_-}(0,1]$, we see that $S_{-i-1} \in \cP_{\sigma_-}(0,2]$.
Now let $F$ be a $\sigma_{-}$-stable object with bigger phase $\phi_{\sigma_{-}}(F) > \phi_{\sigma_{-}}(S^{-}_{-i-1})$, we want to show that $\Hom(F,S_{-i-1})=0$, which will prove $\sigma_-$-semistability of $S_{-i-1}$.

Assume that $F \in \cA_{\sigma_{-}}[n]$ with $n>0$. Applying the functor $\Hom(F,-)$ to the triangle \ref{eq:inductiveTwist} we get the exact sequence:
\begin{equation}\label{eq:exactSeqF}
    \Hom(F,S^{-}_{-i+1}[-1]) \rightarrow \Hom(F,S^{-}_{-i-1}) \rightarrow \Hom(S^{-}_{-i},S^{-}_{-i+1}) \otimes \Hom(F,S^-_{-i}). 
\end{equation}
By induction hypotesis $S^-_{-i+1}$ and $S^-_{-i}$ lie in the heart $\cA_{\sigma_-}$. the first and the third terms vanish because they are negative Ext between objects of a heart, so the middle one does too. 

If $F \in \cA_{\sigma_-}$ is an object of the heart with bigger $\sigma_{-}$-phase, then its class $\ff$ lies under the lower branch of the hyperbola in Figure \ref{lattices}(B), in particular, outside of the closed positive cone, so it has negative square $\ff^2<0$. Since $F$ is stable, it must be spherical, because $\ff^2 \geq -2$. So $F \in \{S_{0},\dots,S^{-}_{-i}\}$. If $F \{S_{0},\dots,S^{-}_{-i+1}\}$ we conclude by induction and the exact sequence \ref{eq:exactSeqF}, so the only case to check is $F=S^{-}_{-i}$. We have
\[\Hom(S_{-i},\ST_{S_{-i}^{-}}(S^{-}_{-i+1})[-1])=\Hom(S^{-}_{-i}[2],S^{-}_{-i+1}) \]
which is again zero because it is a negative Ext between objects of a heart. This shows that $S_{-i-1}$ is $\sigma_{-}$-semistable, since $\phi_{\sigma_-}(T) < \phi_{\sigma_-}(S_{-i-1})< \phi_{\sigma_-}(S)$ it also lies in $\cA_{\sigma}$. Now to show that is $\sigma_-$-stable, consider its Jordan-Holder filtration. Every factor must be a spherical object of the same phase, by Mukai's Lemma. Since the line connecting the origin to $\vv$ meets the lower branch of the hyperbola only in $\vv$, there is only one $\sigma_-$-stable spherical object of that phase up to shifts, so the Jordan-Holder filtration is trivial.
\end{proof}

\subsection{Isotropic Mukai vector}\label{subsec:isotropic}
Now we turn our attention on moduli spaces $M_{\sigma_+}(\vv)$ with vector $\vv$ with $\vv^2=0$. The lattice $\cH$ can be negative semi-definite or hyperbolic, and in the latter case there is only one spherical class up to sign. Indeed, if there were two linearly independent spherical classes, the argument in Remark \ref{twospherical} would show that there could be no isotropic classes. In contrast to the spherical case, where the proof works the same in both cases, if $\vv$ is isotropic the signature of the lattice $\cH$ matters. In Proposition \ref{jhforisotropic} we describe the Jordan-Holder filtration of a stable object $E$ with $v(E)=\vv$ with respect to $\sigma_0$. This result is analogous to \cite[Theorem $12.1$]{bridgeland_stability_2006}, where Bridgeland studies wall-crossing for $\vv=(0,0,1)$ and $\cW$ a wall on the boundary of the geometric chamber $U(X)$. In fact, assuming non-emptiness of moduli spaces with isotropic vector, Proposition \ref{jhforisotropic} follows from Bridgeland's result via a Fourier-Mukai argument, as shown in \cite[Lemma $8.1$]{mmp}.

\begin{Prop}\label{jhforisotropic}
Keeping notation as above we have:
\begin{enumerate}
    \item If $\cH$ is semi-definite then there is a smooth rational $C$ curve inside $M_{\sigma_{+}}(\vv)$ that becomes $\sigma_-$-unstable, and the Jordan-Holder filtration for $E \in C$ with respect to $\sigma_0$ is of the form 
    \[ 0 \rightarrow S \rightarrow E \rightarrow T \rightarrow 0,\]
    where $S$ and $T$ are two $\sigma_0$-stable spherical objects.
    \item If $\cH$ is hyperbolic the wall $\cW$ is totally semistable, i.e. every object $E \in M_{\sigma_+}(\vv)$ becomes $\sigma_0$-semistable. The Jordan-Holder filtration of an object $E \in M_{\sigma_+}(\vv)$ is \[ 0 \rightarrow S^{\oplus a} \rightarrow E \rightarrow F \rightarrow 0 \quad \textrm{ or } \quad  0  \rightarrow F \rightarrow E \rightarrow S^{\oplus a} \rightarrow 0, \]
    where $S$ is a $\sigma_0$-stable spherical object, and $F$ is a $\sigma_0$-stable isotropic object.
\end{enumerate}
Moreover, in both cases the Jordan-Holder filtration with respect to $\sigma_0$ coincides with the Harder-Narasimhan filtration with respect to $\sigma_-$.
\end{Prop}

\begin{proof}
We begin by proving part $(1)$. Since $\cH$ is a negative semi-definite lattice of rank two, the isotropic classes in $\cH_{\R}$ form a one dimensional subspace, which is the radical of the Mukai pairing. Therefore there is at most one, up to a sign, primitive isotropic class in $\cH$. Assume that $E \in M_{\sigma_+}(\vv)$ becomes semistable. From Lemma \ref{Mukailemma} it follows that its Jordan-Holder factors are spherical and isotropic, with at most one being isotropic. 

Since there is just one primitive isotropic class, this means that all the Jordan-Holder factors are spherical, in particular there are two distinct $\sigma_0$-stable spherical objects $S,T$. The only isotropic class is $\ss+\tt$, hence the Jordan-Holder filtration is 
\[ 0 \rightarrow S\rightarrow E \rightarrow T \rightarrow 0,\]
where we assume $\phi_{\sigma_+}(S)<\phi_{\sigma_+}(T).$
We have $\ext^{1}(S,T)=(\ss,\tt)=2$, and every non trivial extension gives a $\sigma_+$-stable object by Lemma \ref{simpleext}. So there is a rational curve $\P(\Ext^1(S,T)) \subset M_{\sigma_+}(\vv)$ of objects that become semistable on the wall. Notice also that the Jordan-Holder filtration with respect to $\sigma_0$ coincides with the Harder-Narasimhan filtration with respect to $\sigma_-$, because $S,T$ are $\sigma_-$-stable with $\phi_{\sigma_-}(S)>\phi_{\sigma_-}(T)$.

Now we prove part $(2)$. First we show the second part of the statement, so let $E$ be $\sigma_0$-semistable. Lemma \ref{Mukailemma} implies that the only objects that can appear as Jordan-Holder factors are spherical and isotropic, with at most one being isotropic. Furthermore, from the discussion in Remark \ref{twospherical}, we see that if $\cH$ contains an isotropic class, then it contains at most one spherical class up to a sign. Therefore there is a unique $\sigma_0$-stable spherical object of the same phase as $E$. Hence all the Jordan-Holder spherical factors are of the form $S^{\oplus a}$. This implies that the Jordan-Holder filtration is of the form
\[ 0 \rightarrow S^{\oplus a} \rightarrow E \rightarrow F \rightarrow 0 \quad \textrm{ or } \quad  0  \rightarrow F \rightarrow E \rightarrow S^{\oplus a} \rightarrow 0, \]
with $F$ isotropic and $\sigma_0$-stable. Which one it is depends on the ordering of the phases: it is the first one if $\phi_{\sigma_+}(S)<\phi_{\sigma_+}(E)$ and the second one if $\phi_{\sigma_+}(E)>\phi_{\sigma_+}(S)$.

As in the previous case, since $S$ and $F$ are $\sigma_0$-stable, they are also $\sigma_-$-stable, so the Harder-Narasimhan filtration coincides with the Jordan-Holder filtration on the wall.

To show that the wall is totally semistable we argue as follows. We have \[\vv^{2}=0=-2a^{2}+2a(\ss,\ww),\] hence $a=(\ss,\ww)$. The spaces $\Hom(S,F)$ and $\Hom(F,S)$ vanish for $\sigma_{0}$-stability, hence $a=\ext^{1}(S,F)$. Applying $\Hom(S,-)$ to the Jordan-Holder filtration we see that $\hom(S,E)=a$ and we get the exact sequence: 
\begin{equation}\label{exactseq}
0\rightarrow \Ext^{1}(S,E) \rightarrow \Ext^{1}(S,F) \rightarrow \Hom(S,S)^{\oplus a} \rightarrow \Hom(E,S).
\end{equation}
By $\sigma_+$-stability the last space is $0$, which implies that $\Ext^{1}(S,F) \cong \Hom(S,S)^{\oplus a}$ because they have the same dimension. Therefore $\Ext^{1}(S,E)=0$, and $\RHom(S,E)=\Hom(S,E)$. This implies $(\ss,\vv)=-\hom(S,E)=-a<0$. In particular for every object $E' \in M_{\sigma_+}(\vv)$ there are non zero morphisms $\Hom(S,E) \neq \emptyset$, so every $E \in M_{\sigma_+}(\vv)$ is $\sigma_0$-semistable.
\end{proof}

We can finish the proof of the main theorem of this section. 

\begin{proof}[Proof of Theorem \ref{semirigid1}(2)]
We separate the proof in two cases, depending on the signature of $\cH$. If $\cH$ is negative semi-definite, we want to show that the spherical twist $\ST_S$ induces an isomorphism $M_{\sigma_+}(\vv) \xrightarrow{\sim} M_{\sigma_-}(\vv)$, where $S$ is the spherical object of Proposition \ref{jhforisotropic}. If $\cH$ is hyperbolic, we want to show that $\ST_S^{\pm 2}$ induces an isomorphism $M_{\sigma_+}(\vv) \xrightarrow{\sim} M_{\sigma_-}(\vv)$, where $S$ is the spherical object of Proposition \ref{jhforisotropic} and the sign depends on the ordering of the phases $\phi_{\sigma_+}(S)$ and $\phi_{\sigma_+}(E)$. 

\textbf{Semi-definite case.} Consider the destabilizing spherical object $S$ of Proposition \ref{jhforisotropic}. We claim that
\begin{enumerate}
    \item If $E \in M_{\sigma_+}(\vv)$ remains stable on the wall, then $\ST_S(E)=E$,
    \item If $E \in M_{\sigma_+}(\vv)$ becomes semistable on the wall, then $\ST_S(E)$ is $\sigma_-$-stable.
\end{enumerate}
To show $(1)$ observe that if $E$ remain stable on the wall, then $\sigma_0$-stability gives $\Hom(E,S)=\Hom(S,E)=0$. Since $(\ss,\vv)=0$, we also get $\RHom(S,E)=0$. It follows from the definition of spherical twist that then $\ST_{S}(E)=E$.

To show $(2)$, consider the Jordan-Holder filtration 
\[0 \rightarrow S \rightarrow E \rightarrow T \rightarrow 0.\]
Applying $\Hom(S,-)$ to the Jordan-Holder filtration we see that $\Hom(S,E) \cong \Hom(S,S)$ is one dimensional. By $\sigma_{+}$-stability we get $\Hom(E,S)=0$, and from $(\ss,\vv)=0$ we see $\ext^{1}(S,E)=1$. The definition of spherical twist gives a distinguished triangle:
\[ S \oplus S[-1] \rightarrow E \rightarrow \ST_{S}(E) \rightarrow S[1] \oplus S \]
Taking the long exact sequence with respect to the heart $\cA_{\sigma_{0}}$ we get the two short exact sequences:
\begin{align*}
    &0\rightarrow S \rightarrow E \rightarrow R \rightarrow 0\\
    &0\rightarrow R \rightarrow \ST_{S}(E) \rightarrow S \rightarrow 0
\end{align*}
The first one shows $R=T$ and it is the Jordan-Holder filtration. The second one then becomes $T \rightarrow \ST_{S}(E) \rightarrow S$ which shows that $\ST_{S}(E)$ is $\sigma_{-}$-stable using Lemma \ref{simpleext}. Starting from $F \in M_{\sigma_-}(\vv)$, the filtration is reversed, and the analogous argument shows that $\ST_{S}^{-1}(F)$ is $\sigma_+$-stable. In conclusion, passing to moduli spaces we see that $\ST_S$ induces an isomorphism
$ M_{\sigma_+}(\vv) \xrightarrow{\sim} M_{\sigma_-}(\vv).$

{\textbf{Hyperbolic case.}} Let $E \in M_{\sigma_+}(\vv)$, Proposition \ref{jhforisotropic} gives the Jordan-Holder filtration with respect to $\sigma_0$:
\[ 0 \rightarrow S^{\oplus a} \rightarrow E \rightarrow F \rightarrow 0.\]
Applying $\Hom(S,-)$ to it we get the exact sequence \eqref{exactseq}. The steps in the proof of Proposition \ref{jhforisotropic} also show that the map $S^{\oplus a} \rightarrow E$ in the Jordan-Holder filtration has the universal property of the evaluation map, hence $F$ is canonically isomorphic to $\ST_S(E)$.

Now, we have the two distinguished triangles
\begin{align*}
 &\Hom(S,E) \otimes S \rightarrow E \rightarrow \ST_{S}(E)\rightarrow \Hom(S,E) \otimes S[1]\\
 &\ST_{S}(E) \rightarrow \ST^{2}_{S}(E) \rightarrow \Hom(S,E) \otimes S \rightarrow \ST_{S}(E)[1],
\end{align*}
where the first one is obtained by definition, and the second one applying $\ST_S$ to the first.
Since $\ST_{S}(E) \cong F$, we conclude that $\ST^{2}_{S}(E)$ is $\sigma_{-}$-stable by Lemma \ref{simpleext}. An analogous argument shows that $\ST^{-2}_S$ sends $\sigma_-$-stable objects with vector $\vv$ to $\sigma_+$-stable objects with vectors $\vv$. Passing to moduli spaces we see that $\ST_S^2$ induces an isomorphism
$ M_{\sigma_+} \xrightarrow{\sim} M_{\sigma_-}.$
\end{proof}

\begin{Cor}\label{yoshiokaez}
Let $X$ be a K3 surface, $\vv \in \Halg$ primitive with $\vv^2 =-2$ or $\vv^2=0$. Let $\cW \subset \Stab^{\dagger}(X)$ be a wall, $\sigma_0 \in \cW$ a generic stability condition on the wall, and $\sigma_{\pm}$ generic stability conditions on the adjacent chambers. Then the Main Theorem holds for $M_{\sigma_+}(\vv)$ if and only if it holds for $M_{\sigma_-}(\vv)$. 
\end{Cor}

\begin{proof}
If $\vv^2=-2$ we have to show that if $M_{\sigma_+}(\vv)$ is a single point, the same is true is $M_{\sigma_-}(\vv)$. Point $(1)$ of Theorem \ref{semirigid1} gives $M_{\sigma_-}(\vv) \neq \emptyset$, so we only have to show uniqueness of stable spherical objects with class $\vv$. Assume that there are two non isomorphic spherical stable objects $E,E'$ with the same vector $\vv$. Up to shift we can assume they are both in the heart of $\sigma_{-}$. By stability, we have $\Hom(E,E')=\Hom^{2}(E,E')=0$. Since they are in the heart, we have $0 \leq \hom^1(E,E')=\vv^{2}=-2$, which is a contradiction. 

If $\vv^2=0$ and primitive, we want to show that if $M_{\sigma_+}(\vv)$ is a K3 surface, and \[ \theta_{\vv}:\vv^{\perp}/\Z \vv \rightarrow H^2(M_{\sigma_+}(\vv),\Z) \] is a Hodge isometry, the same is true for $M_{\sigma_-}(\vv)$. This follows from part $(2)$ of Theorem \ref{semirigid1} combined with Proposition \ref{prop:equivalences}.
\end{proof}

\section{Wall-crossing for the Hilbert Scheme}
\label{subsec:Positivesquare}

In this section we study wall-crossing for the Hilbert scheme of $n$ points on a K3 surface of Picard rank one when the degree is high with respect to the number of points. Of course, this setting is less general than the previous one; nevertheless thanks to the argument in Section \ref{sec:Deformation} we will be able to reduce to this case. Throughout this section we assume $X$ is a K3 surface with $\Pic(X)=\Z \cdot H$ with $H^2=2d$ and $d=k^2(n-1)$, where $k \in \Z,k>1$, the Mukai vector is $\vv=(1,0,1-n)$, and the stability condition is $\sigma_{\alpha,\beta} \in V(X)$. The notation $\sigma_{\alpha,\beta}$ denotes the stability condition $\sigma_{\alpha H,\beta H} \in V(X)$, see Section \ref{subsec:stabilityK3} for the definition. In particular, the heart is $\Coh^{\alpha H,\beta H}(X)$ and the central charge is
\[ Z_{\alpha,\beta}(\vv):=(\vv,\exp(\beta H+i\alpha H))=dr(\alpha^2-\beta^2)+2dc\beta-s+2id(c-r\beta)\alpha.\]
One can check that the heart does not depend on $\alpha$; in this section we will denote it just by $\Coh^{\beta}(X)$.
It is a stability condition for $(\beta,\alpha) \in \R \times \R_{>0} $, provided that $\Im Z(E)\neq 0$ for every spherical torsion-free sheaf $E$. Hence, the domain $V(X)$ is identified with the upper half plane $\R \times \R_{>0}$ with some isolated points removed. The following is the main result of this section.

\begin{Thm}\label{thm:hilbert_scheme}
Let $X$ be a K3 surface with $Pic(X)=\Z \cdot H$ with $H^2=2d$ and $d=k^2(n-1)$ for $k>1$ integer. Then $M_{\sigma_{\alpha,\beta}}(\vv) = \Hilb^n(X)$ for every $\sigma_{\alpha,\beta} \in V(X)$ with $\beta <0$, and $\vv=(1,0,1-n)$.
\end{Thm} 

For convenience in this section we work with the slope $\nu_{\alpha,\beta}$, instead of the phase $\phi_{\alpha,\beta}$. It is defined for objects $E \in \Coh^{\beta}(X)$ as 
\[ \nu_{\alpha,\beta}(E):=
\begin{cases}
-\frac{\Re Z_{\alpha,\beta}(\uu)}{\Im Z_{\alpha,\beta}(\uu)}=\frac{dr(\beta^2-\alpha^2)-2dc\beta+s}{2d(c-r\beta)\alpha} & \textrm{if } \Im Z_{\alpha,\beta}(\uu) \neq 0\\
+\infty & \textrm{if } \Im Z_{\alpha,\beta}(\uu) = 0\\
\end{cases}
,\]
where $\uu=v(E)=(r,cH,s)$. 
It is related to the phase via $\phi_{\alpha,\beta}(E)=\frac{1}{\pi}\cot^{-1}(\nu_{\alpha,\beta}(E))$, so it gives the same notion of stability on $\Coh^{\beta}(X)$. 
If $\cW$ is a wall for $\vv$, and $\uu$ is the class of a destabilizing object, we can recover the equation of the wall by $\nu_{\alpha,\beta}(\uu)=\nu_{\alpha,\beta}(\vv)$. Writing $\uu=(r,cH,s)$ and expanding this equation we get
\[ cd(\alpha^2+\beta^2)-\beta(r(n-1)+s)+c(n-1)=0.\]
The following is a well know fact, see \cite[Proposition $3.7$]{macri2020stability}.

\begin{Thm}[Bertram Nested Wall Theorem]
Let $\vv=(r,cH,s) \in \Halg$, with $\vv^2>0$. The walls in $V(X)$ are either semicircles with center in the $\beta$-axis or lines parallel to the $\alpha$-axis. If $r\neq 0$ there is a unique vertical wall at $\beta=\mu(\vv)$, and there are two sets of nested semicircular walls, one on each side of the vertical wall. 
\end{Thm}

We are interested in the walls for the vector $\vv=(1,0,1-n)$ of the Hilbert scheme. In this case the above Theorem tells us that there is a vertical wall on the line $\beta=0$; this wall corresponds to the Hilbert-Chow contraction. The moduli space $M_{\alpha,\beta}(\vv)$ with $\alpha >>0$ and $\beta <0$ is the Hilbert scheme of points $\Hilb^n(X)$, thanks to Theorem \ref{large_volume}. It parameterizes ideal sheaves of subschemes $Y \subset X$ of dimension 0 and lenght $n$. On the vertical wall two ideal sheaves become $S$-equivalent if and only if the corresponding subschemes have the same support. 

Recall that to a wall $\cW$ we associate the rank two lattice $\cH$ given by all the classes $\uu$ with $Z(\uu)$ on the same line of $Z(\vv)$, see Definition \ref{def:lattice}. Lemma \ref{signature} implies that $\cH$ is hyperbolic, since it contains the class $\vv$, which has positive square by assumption. 

The idea of the proof of Theorem \ref{thm:hilbert_scheme} is simple. First, we show that the vertical line $\beta=-\frac{1}{k}$ does not meet any semicircular wall. This is because the imaginary part $\Im Z_{t,-\frac{1}{k}}(E)$ takes non-negative integer values (up to a constant) when $E$ varies in $\Coh^{\beta}(X)$, and $\Im Z_{t,-\frac{1}{k}}(\vv)$ is the minimal positive value. This is completely analogous to the fact that rank one torsion free sheaves are Gieseker stable with respect to any polarization. 

Then we show that, in fact, the line $\beta=-\frac{1}{k}$ meets every semicircular wall in the left quadrant. This implies that there are no semicircular walls in the left quadrant. 

\begin{Lem}[{{\cite[Examples $9.7$ and $10.5$]{bayer_projectivity_2013}}}]
The stability condition $\sigma_{t,-\frac{1}{k}}$ exists for every $t>0$, and the moduli space $M_{\sigma_{t,-\frac{1}{k}}}(1,0,1-n)$ is equal to the Hilbert scheme $\Hilb^n(X)$ 
\end{Lem}

\begin{proof}
We have 
\begin{equation}\label{eq:imaginaryPart}
\Im Z_{t,-\frac{1}{k}}(r,cH,s)=2dt\frac{ck+r}{k} \in \frac{2td}{k}\Z, 
\end{equation}
for any vector $(r,cH,s)$. First we show that the stability condition $\sigma_{t,-\frac{1}{k}}$ is defined for every $t>0$. This means that there is no spherical class $\uu=(r,cH,s)$ such that $\Im Z_{t,-\frac{1}{k}}(\uu)=0$. If there were one, it would satisfy 
\[
\begin{cases}
r=-ck\\
2dc^2=2rs-2
\end{cases}
\]
Substituting $d=k^2(n-1)$ gives a contradiction with $k>1$. 

Now assume that an object $E$ of class $(1,0,1-n)$ becomes semistable for some stability condition $\sigma_{t,-\frac{1}{k}}$. We have a destabilizing short exact sequence in $\Coh^{\beta}(X)$
\[ 0 \rightarrow F \rightarrow E \rightarrow G \rightarrow 0,\]
with $\nu_{t,-\frac{1}{k}}(F)=\nu_{t,-\frac{1}{k}}(E)=\nu_{t,-\frac{1}{k}}(G)<+\infty.$
By definition of Bridgeland stability condition, we have  
\[ 0 \leq \Im Z_{t,-\frac{1}{k}}(F) \leq \Im Z_{t,-\frac{1}{k}}(E)=\frac{2dt}{k},  \]
and similarly for $G$. 
The equality \eqref{eq:imaginaryPart} implies that $\Im Z_{t,-\frac{1}{k}}(F)=0$ or $\Im Z_{t,-\frac{1}{k}}(G)=0$. In both cases this contradicts the finiteness of the slopes. 
\end{proof}

\begin{Lem}
The vertical line $\beta=-\frac{1}{k}$ meets every semicircular wall in the left quadrant $\beta<0$.
\end{Lem}

\begin{proof}
Consider a destabilizing exact sequence on the wall $\cW$:
\[ 0 \rightarrow F \rightarrow E \rightarrow G \rightarrow 0.\]
The equation of the wall is $\nu_{\alpha,\beta}(F)=\nu_{\alpha,\beta}(E)$. Call $\uu=v(F)$ and $\vv=(1,0,1-n)$. To simplify the computations we change $\uu$ in the lattice $\cH$ with a vector of rank zero, such that the equation of the wall is still given by $\nu_{\alpha,\beta}(\vv)=\nu_{\alpha,\beta}(\uu)$. To do this is sufficient to take $\uu$ as the Mukai vector of a semistable object in the heart $\Coh^{\beta}(X)$ of the same slope as $E$, we do it as follows. 
\begin{itemize}
    \item If $r=0$, we do not change $\uu$.
    \item If $r>0$ take $\uu:=v(E^{\oplus{r-1}}\oplus G)$.
    \item If $r<0$ take $\uu:=v(E^{\oplus -r} \oplus F)$.
\end{itemize}
Write $\uu=(0,cH,s)$, the equation of the wall and the $\beta$ coordinate of the center become: 
\begin{equation}
\begin{cases}
cd(\alpha^2+\beta^2)-\beta s+c(n-1)=0,\\
\beta_0=\frac{s}{2cd}.
\end{cases}
\end{equation}
Since $\uu$ is the class of an object in the heart $\Coh^{\beta}(X)$, it satisfies 
\[ \Im Z_{\alpha,\beta}(\uu) = 2dc\alpha \geq 0. \]
Since $\alpha >0$ this gives $c \geq 0$. If $c=0$ we would get the Hilbert-Chow wall, so we have $c > 0$. The center of any semicircular wall is on the negative $\beta$-axis, hence $\beta_0<0$. The above equation gives $s < 0$.

The lattice $\cH$ is hyperbolic, and $\uu$ and $\vv$ are linearly independent over $\R$, so 
\[ 
\det \left(
\begin{array}{cc}
    \vv^2 & (\vv,\uu)  \\
     (\uu,\vv) & \uu^2 
\end{array}
\right)<0,
\]
where
\[ 
\begin{cases}
\vv^2=2(n-1), \\
\uu^2=2dc^2, \\
(\uu,\vv)=-s.
\end{cases}
\]
This implies 
\[ 
4d(n-1)c^2 -s^2 <0.
\]
Substituting $d=k^2(n-1)$, and taking square roots we get:
\[ s<-2k(n-1)c \quad \textrm{or} \quad  s>2k(n-1)c.\]
The second inequality contradicts $s<0$ and $c>0$, so we must have 
\begin{equation}\label{eq:hyperbolictyInequality}
    s<-2k(n-1)c
\end{equation}

The condition for the wall to meet the vertical line $\beta=-\frac{1}{k}$ is for the equation
\[ cd(\alpha^2+\frac{1}{k^2})+\frac{s}{k}+c(n-1)=0,\]
to have a solution for $\alpha>0$. Substituting $d=k^2(n-1)$ and rearranging we get:
\[ ck^2(n-1)\alpha^2=-\frac{s}{k}-2(n-1)c.\]
Since $c>0$ a solution exists if and only if
$\frac{s}{k}+2(n-1)c<0,$
which is \eqref{eq:hyperbolictyInequality}.
\end{proof}

\begin{Cor}\label{Cor:unigonal}
Let $X$ be a K3 surface with $Pic(X)=\Z \cdot H$ with $H^2=2d$ and $d=k^2(n-1)$ for $k>1$ integer. The vertical wall $\beta=0$ is the only wall for $\vv=(1,0,1-n)$ in $V(X)$, and the shifted derived dual $\RcHom(-,\cO_X)[2]$ induces an isomorphism
\[ M_{\sigma_{\alpha,\beta}}(\vv) \xrightarrow{\sim} M_{\sigma_{\alpha,-\beta}}(\vv),\]
for $\beta \neq 0$.
In particular, the Main Theorem holds for both of them.
\end{Cor}

\begin{proof}
The functor $\RcHom(-,\cO_X)[2]$ induces the desired isomorphism by \cite[Proposition $2.11$]{mmp}. Combined with Theorem \ref{thm:hilbert_scheme} this implies that $\beta=0$ is the only wall in $V(X)$. The moduli space $M_{\sigma_{\alpha,\beta}}(\vv)$ for $\beta<0$ is the Hilbert scheme, so the Main Theorem holds for it (Example \ref{ex:hilbert_scheme}) and $\RcHom(-,\cO_X)[2]$ is an anti-autoequivalence, so it preserves the Mukai homomorphism by Proposition \ref{prop:equivalences}.
\end{proof}

\section{Reduction to the Hilbert scheme}
\label{sec:Deformation}

In this section we conclude the proof of the Main Theorem. We fix a K3 surface $X$, a primitive Mukai vector $\vv \in \Halg$, with $\vv^{2} \geq -2$, and a $\vv$-generic stability condition $\sigma \in \Stab^{\dagger}(X)$. Recall the statement of the Main Theorem

\begin{Thm}
Let $X$ be a K3 surface, $\vv \in \Halg$ primitive, and $\sigma\in\Stab^\dagger(X)$
a $\vv$-generic stability condition.
Then:
\begin{enumerate}[{\rm (1)}]
\item $M_{\sigma}(\vv)$ is non-empty if and only if $\vv^2\geq-2$.
Moreover, it is a smooth projective hyperk\"{a}hler variety of dimension $\vv^2 + 2$, deformation-equivalent to the Hilbert scheme of points on a K3 surface.
\item If $\vv^2\geq 0$, then the Mukai Homomorphism $\theta_{\vv}$ gives a Hodge isometry
\[
\theta_{\vv}^{-1}\colon H^2(M_\sigma(\vv),\Z)\xrightarrow{\quad\sim\quad}
\begin{cases}\vv^\perp & \text{if }\vv^2>0\\ \vv^\perp/\Z\vv & \text{if } \vv^2=0,\end{cases}
\]
where the orthogonal is taken in $H^{*}(X,\mathbb{Z})$.
\end{enumerate}
\end{Thm}

The proof is in several steps. First, we apply a sequence of autoequivalences to get a Mukai vector of the form $\vv=(r,\Delta,s)$ with $r>0$ and $\Delta$ ample. Since $\Delta$ is ample, the Hodge locus of $\vv$ contains an ellitpic K3 surface $X'$ with a section. We deform to $X'$, where we can find a vector of the form $\ww=(0,\alpha f,\beta)$, where $f$ is the class of an elliptic fiber, such that $(\ww,\vv)=-1$. The moduli space $M=M_H(\ww)$ is non-empty: a generic point is just a vector bundle supported on a smooth fiber. Moreover, it is a fine moduli space, and the Fourier-Mukai $\Phi_{\cE}$ with the universal family as kernel is an equivalence $D^b(X') \xrightarrow{\sim} D^b(M)$. Via this equivalence $\vv$ gets mapped to $(1,0,1-n)$, up to tensoring with line bundles. Now assume $n \leq 1$ or equivalently $\vv^2 = -2$ or $\vv^2=0$. The moduli space of Gieseker stable shaves with vector $(1,0,0)$ is a point, and with vector $(1,0,1)$ is the K3 surface itself. To conclude the proof of the Main Theorem in this case we apply the wall-crossing results of Section \ref{sec:Wall crossing}.

If $\vv^2>0$, we take a different K3 surface as a Fourier-Mukai partner: $M=M_{\sigma}(\ww)$, where $\ww$ is the same vector as before, but the stability condition is the same one we are studying. We are allowed to do this, because we proved the Main Theorem for isotropic vectors first. As before, via the Fourier-Mukai transform $\Phi_{\cE}$ the vector $\vv$ goes to $(1,0,1-n)$, but $\sigma$ goes to the geometric chamber $U(X)$. Moreover, we show that $\Pic(M)$ is an hyperbolic plane, so we can deform to a K3 surface $Y$ of the type studied in Section \ref{subsec:Positivesquare}. Being in the geometric chamber is an open condition, so the deformed stability condition remains in the geometric chamber for $Y$. If we act with $\widetilde{GL_2(\R)}$ we end up in the setting of Section \ref{subsec:Positivesquare}, where the moduli space is just the Hilbert scheme $\Hilb^n(Y)$ up to a shifted derived dual.

In the argument, we apply Proposition \ref{prop:equivalences} to equivalences of type $(1)-(4)$ of Section \ref{subsec:preserveStab}, which preserve the distinguished component (Corollary \ref{preserve}). It is useful to recall their action in cohomology:
\begin{enumerate}
    \item Tensor product with $L \in \Pic(X)$ acts via multiplication with $\exp(c_1(L))$
    \[ \big(r,\Delta,s\big) \left(1,c_{1}(L),\frac{c_{1}(L)^2}{2}\right)=\left(r,\Delta+rc_{1}(L),r\frac{c_{1}(L)^2}{2}+\Delta.c_{1}(L)+s\right).\]
    \item The shift $[1]$ acts as $-\id$.
    \item The spherical twists acts like the reflection around $(1,0,1)$
    \[ \rho_{(1,0,1)}((r,\Delta,s))=(-s,\Delta,-r).\]
    \item The Fourier-Mukai $\Phi_{\cE}$ acts like the cohomological Fourier-Mukai, with kernel the Mukai vector $v(\cE)$.
\end{enumerate}

For the deformation arguments we use the notion of a relative stability condition $\underline{\sigma}$ over a base $C$. It was introduced in \cite{bayer2019stability}, and it consists, given a family $\cX \rightarrow C$, of a collection $\underline{\sigma_c}$ of stability conditions on the fibers $\cX_c$ satisfying some technical conditions. There is also a well-behaved notion of relative moduli space. The following is the result we use, it is stated in \cite[Corollary $32.1$]{bayer2019stability} for cubic fourfolds, but the same proof works for polarized K3 surfaces. 

\begin{Thm}\label{families}
Let $(X,H)$ be a polarized K3 surface of degree $2d$, $\vv$ a primitive vector, and $\sigma \in \Stab^{\dagger}(X)$ a $\vv$-generic stability condition . Let $(X',H')$ be another polarized K3 surface of the same degree, in the Hodge locus where $\vv$ stays algebraic inside the moduli space of polarized K3 surfaces of degree $2d$. Then, there exists a smooth family $\cX \rightarrow C$ over a smooth connected quasi-projective curve, and a stability condition $\underline{\sigma} \in D^{b}{(\cX)}$ such that:
\begin{enumerate}
    \item The class $\vv$ stays algebraic for all $c \in C$.
    \item The stability condition $\underline{\sigma}_{c}$ is in $\Stab^{\dagger}(\cX_c)$ and $\vv$ generic for all $c \in C$.
    \item $\cX_{c_{0}}=X$, $\cX_{c_1}=X'$ and $\underline{\sigma}_{c_0}$ is a small deformation of $\sigma$ such that $M_{X,\sigma}(\vv)=M_{X,\sigma_{c_0}}(\vv)$.
    \item The relative moduli space $M_{\underline{\sigma}}(\vv)$ exists as a smooth and proper algebraic space over $C$. 
\end{enumerate}
\end{Thm}

The first step in the reduction to the Hilbert scheme is to apply a sequence of autoequivalences to change the Mukai vector $\vv$.

\begin{Lem}\label{change_vector}
Let $X$ be a K3 surface, $\vv=(r,\Delta,s)$ primitive, and let $\sigma \in \Stab^{\dagger}(X)$ be $\vv$-generic. Then, there exist a primitive $\vv'=(r',\Delta',s')$ with $r'>0$ and $\Delta'$ ample, a $\vv'$-generic stability condition $\sigma' \in \Stab^{\dagger}(X')$, and an isomorphism $M_{\sigma}(\vv) \cong M_{\sigma'}(\vv')$.
\end{Lem}
\begin{proof}
First we reduce to $r>0$. If $r<0$ then a shift suffices. If $r=0$ and $\Delta=0$, then $\vv=(0,0,\pm1)$ so after applying either $\ST_{\cO_X}[1]$ or $\ST_{\cO_X}$ we get $(1,0,0)$. If $\Delta \neq 0$, then, up to a shift, we can assume it to be effective. If $H$ is an ample line bundle, tensor product with $nH$ sends $\vv$ to $(0,\Delta,s+nH.\Delta)$. By taking $n>>0$ we can assume $s>0$. Applying the shifted spherical twist $\ST_{\cO_X}[1]$ we get $r>0$. 

If $r>0$, to get a $\Delta$ ample we can tensor with powers of an ample line bundle. Indeed, $\Delta$ goes to $\Delta + rnH$, which is ample if $n>>0$. The distinguished component is preserved due to Proposition \ref{prop:equivalences}. 
\end{proof}

The next step is a deformation to an elliptic K3 surface. Consider $(X,\vv,\sigma)$ as in the conclusion of the lemma above, i.e. $\vv=(r,\Delta,s)$ with $r >0 $ and $\Delta$ ample. We write $\Delta=mH$ with $m \in \Z_{>0}$ and $H$ a primitive polarization on $X$ of degree $H^2=:2d$ 

\begin{Lem}\label{first_deformation}
Let $(X,\vv,\sigma)$ be as in the conclusion on the lemma above. Then, there exists an elliptic K3 surface $X'$ in the hodge locus of $\vv$, with $\Pic(X')=\Z s \oplus \Z f$, where $f$ is the class of the elliptic fiber and $s$ the class of a section, and a stability condition $\sigma'$ on $X'$ such that $M_{X,\sigma}(\vv)$ is deformation equivalent to $M_{X',\vv'}(\vv)$. 
\end{Lem}
\begin{proof}
We only need to check the hypothesis of Theorem \ref{families}. First assume $d>1$. By the surjectivity of the period map, there exists a K3 surface with Picard group as in the statement. Equipping it with the polarization $s+(d+1)f$, it defines a point in the moduli space of polarized K3 surfaces of degree $2d$. Since $\Delta$ is a multiple of a polarization $H$ it remains algebraic on $Y$, so we are in the hypotesis of Theorem \ref{families}. The relative moduli space of point $(4)$ gives us the desired deformation. 

If $d=1$, the class $s+(d+1)f$ is not ample. In this case, we can apply Theorem \ref{families} to first deform to a K3 surface $X''$ with $\rho(X'')>1$. Indeed, such K3 surfaces are dense in the Hodge locus of $\vv$. On $X''$ we can tensor by an ample line and obtain a Mukai vector $\vv''=(r'',\Delta'',s'')$ with $\Delta''=mH''$ and $(H'')^2>>0$. So we reduced to the case $d>1$, and the argument above concludes the proof of the Lemma.
\end{proof}

\begin{Rem}\label{rem:defMukai}
Since the previous deformation is given by a relative moduli space, the quasi-universal family deforms, and so does the Mukai homomorphism. In particular, the function 
\[ \theta_{\vv}: \vv^{\perp} \cap H^{*}(\cX_c,\Z) \rightarrow H^{2}(M_{\cX_c,\sigma_c}(\vv),\Z)\]
is a locally constant on $C$. Since $C$ is connected, and the Beaville-Bogomolov form is deformation invariant, $\theta_{\vv}$ is a Hodge isometry on $X'$ if and only if it is on $X$.
\end{Rem}

Now we prove the Main Theorem for spherical and isotropic classes $\vv$. 

\begin{Thm}\label{Yoshioka_semirigido}
Let $X$ be a K3 surface, $\vv$ primitive and $\sigma \in \Stab^{\dagger}$ generic. If $\vv^{2}=-2$, the moduli space $M_{\sigma}(\vv)$ is a reduced point. If $\vv^{2}=0$, the moduli space $M_{\sigma}(\vv)$ is a projective K3 surface, and the map $\theta_{\vv}: \vv^{\perp}/\Z \vv \rightarrow H^{2}(M_{\sigma}(\vv),\Z)$ is a Hodge isometry. 
\end{Thm}
\begin{proof}
As a preliminary remark, notice that if $\vv^2=0$, then $M_{\sigma}(\vv)$ is a two-dimensional smooth and proper algebraic space, hence projective, and moreover it is symplectic. So, to prove the Theorem it is enough to show that $M_{\sigma}(\vv)$ is deformation equivalent to a point if $\vv^2=-2$ or to a K3 surface if $\vv^2=0$.

From Lemmas \ref{change_vector} and \ref{first_deformation} we can assume that $X$ is an elliptic K3 surface, with $\Pic(X)=\Z s \oplus \Z f$, where $f$ is the class of a fiber and $s$ is the class of a section. Moreover, from Lemma \ref{coprime} and an application of the shifted spherical twist $\ST_{\cO_X}[1]$ we can assume that the Mukai vector $\vv=(r,\Delta,s)$ has rank positive rank $r>0$ and coprime with $\Delta.f$. 

Consider a vector $\ww=(0,\alpha f,\beta)$; we have $(\ww,\vv)=\alpha \Delta.f - \beta r$. Since $r$ and $\Delta.f$ are coprime, we can find $\alpha$ and $\beta$ such that $(\vv,\ww)=-1$. Since $r>0$ we can assume also that $\alpha >0$ and $\beta \neq 0$. Let $H$ be a polarization such that $M:=M_{H}(w)$ is non-empty and parameterizes stable sheaves, as in Example \ref{ex:Fm_partner}. The moduli space $M$ is fine because $(-\vv,\ww)=1$, see \cite[Remark $4.6.8$]{huybrechts_lehn_2010}. 

Consider the Fourier-Mukai transform 
\[\Phi_{\cE}:D^{b}(M) \xrightarrow{\sim} D^{b}(X)\] 
given by the universal family; it is an equivalence by Proposition \ref{prop:FMequivalence}. Since $M$ is a projective symplectic surface, derived equivalent to a K3 surface, it is a K3 surface itself. At the level of cohomology $\Phi_{\cE}^{H}$ is an isometry, and $(\Phi^H_{\cE})^{-1}(\ww)=(0,0,1)$. Define $\sigma':=\Phi_{\cE,*}^{-1}(\sigma)$ and $\vv':=(\Phi^H_{\cE})^{-1}(\vv)$. Then \[r(\vv')=-(\vv',(0,0,1))=-(\vv,\ww)=1\]
Up to twisting for a line bundle we can assume $\vv'=(1,0,1-n)$, with $n=\frac{\vv^{2}+2}{2} \geq 0$. By Proposition \ref{prop:equivalences} the moduli space $M_{X,\sigma}(\vv)$ is isomorphic to the moduli space $M_{M,\sigma'}((1,0,1-n))$, and the morphism $\theta_{\vv}$ is compatible with this isomorphism.

Consider the wall and chamber decomposition for the vector $\vv'$ on $\Stab^{\dagger}(M)$. Let $H'$ be a polarization on $M$. From Theorem \ref{large_volume} there is a chamber where Bridgeland stability is the same as Gieseker $H'$-stability. If $\vv^{2}=-2$, then $\vv'=(1,0,1)$. The moduli space for this vector in the Gieseker chamber is a reduced point corresponding to $\cO_{X}$. If $\vv^{2}=0$, the new vector is $(1,0,0)$ and the moduli space in the Gieseker chamber parameterizes ideal sheaves of points, so it is isomorphic to the underlying K3 surface. Moreover, the Mukai homomorphism is just the identity on $H^2(X,\Z)$.

Since $\Stab^{\dagger}(M)$ is connected by definition, we can find a path that connects the Gieseker chamber with the stability condition $\sigma'$. This will intersect finitely many walls, because they are locally finite. From Corollary \ref{yoshiokaez} and Remark \ref{rem:defMukai} we get the thesis. 
\end{proof}

\begin{Lem}\label{coprime}
Let $X$ be an elliptic K3 surface with $\Pic(S)=\Z s \oplus \Z f$ where $f$ is the class of an elliptic fiber, and $s$ is the class of a section. Let $\vv=(r,m(s+df),s)$ primitive, with $d > 0$ and $r>0$. Then, there exists $k$ such that $\vv.\ch(kf)=(r,\Delta_{k},s_{k})$ has $\gcd(\Delta_{k},s_{k})=1$.
\end{Lem}

\begin{proof}
We have 
\begin{align*}
\Delta_{k}&=ms + (md+kr)f;\\
s_{k}&=s+mk.
\end{align*}
In particular $\gcd(s_{k},m)=\gcd(s,m)$. Since $s$ and $f$ are primitive we get
\[  \gcd(s_{k},\Delta_{k})=\gcd(\gcd(s_{k},m),\gcd(s_{k},md+rk)) \mid \gcd(\gcd(s,m),md+rk).\]
Since $\vv$ is primitive, we have $\gcd(\gcd(s,m),\gcd(md,r))=1$, because no prime can divide $\gcd(s,m)$ and $r$. Call $c=\gcd(md,r)$, and write $md+kr=c(\frac{md}{c}+k\frac{r}{c})$. 

By the Dirichlet Theorem on arithmetic progressions, we can find $k$ such that $(\frac{md}{c}+k\frac{r}{c})$ is a prime bigger than $\gcd(s,m)$. This implies $md+kr$ is coprime with $\gcd(s,m)$ hence the thesis.
\end{proof}

The last step is to prove the Main Theorem for Mukai vector $\vv$ such that $\vv^2>0$. We first show that the Picard group of Fourier-Mukai partner $M$ is an hyperbolic plane, and then deform to a K3 surface of Picard rank one. 

\begin{Lem}\label{fm}
Let $X$ be an elliptic K3 surface with $\Pic(X)=\Z s \oplus \Z f$, let $\vv=(r,m(s+(d+1)f),s)$ primitive, with $r>0$ and $\sigma$ generic. There exists another elliptic K3 surface $M$ with $\Pic(M)=\Z s'\oplus \Z f'$, and an isomorphism $M_{X,\sigma}(\vv)\cong M_{M,\sigma'}((1,0,1-n))$ where $n=\frac{\vv^{2}+2}{2}$ and $\sigma' \in U(M)$ is generic for $(1,0,1-n)$.
\end{Lem}
\begin{proof}
We begin as in the proof of Theorem \ref{Yoshioka_semirigido}: we apply Lemma \ref{coprime} and a spherical twist to reduce to $r$ and $\Delta.f$ coprime, and we consider a vector $\ww=(0,\alpha f,\beta)$ such that $(\ww,\vv)=-1$. Deforming $\sigma$ if necessary, we can assume it to be $\ww$-generic too. Theorem \ref{Yoshioka_semirigido} applied to the moduli space $M:=M_{\sigma}(\ww)$, implies that it is non-empty and a K3-surface. It is fine because wall-crossing preserves the universal family, and the universal family induces a derived equivalence
\[\Phi_{\cE}: D^b(M) \xrightarrow{\sim} D^b(X). \]
As in the proof of Theorem \ref{Yoshioka_semirigido}, define 
\[\sigma':=\Phi_{\cE,*}^{-1}(\sigma)\  \mathrm{and} \  \vv':=(\Phi^H_{\cE})^{-1}(\vv).\]
Up to twisting with a line bundle on $M$ we can assume $\vv=(1,0,1-n)$. To conclude the proof, it remains to show that $\sigma'$ is in $U(M)$ and that $\Pic(M)=\Z s' \oplus \Z f'$.

From Lemma \ref{UandV} and Proposition \ref{prop:orientation}, we only have to show that the skyscraper sheaves $\{\cO_{m} \mid m\in M \}$ are $\sigma'$ stable. This is true because $\Phi_{\cE}(\cO_{m})$ are precisely the objects of the moduli space $M$, which by construction are $\sigma$-stable.

For the second statement, consider the two vectors $\ww':=(\alpha,\beta s + (\alpha +\beta)f,\beta)$ and $\tt:=(\alpha,\beta s + (\beta-r)f,-\Delta.f)$ on $X$, where $\beta r - \alpha \Delta.f=1$. It is a computation to check that they satisfy the following relations.
\[
\begin{cases}
(\ww',\ww')=0,\\
(\ww',\ww)=0.
\end{cases}
\qquad
\begin{cases}
(\tt,\tt)=-2,\\
(\tt,\ww)=0,\\
(\tt,\ww')=-1.
\end{cases}
\]
This implies that $(\Phi^{H}_{\cE})^{-1}(\ww')=(0,l,a)$ and $(\Phi^{H}_{\cE})^{-1}(\tt)=(0,t,b)$, with 
\[
\begin{cases}
(l,l)=0,\\
(t,t)=-2,\\
(l,t)=-1.
\end{cases}
\]
which means that $\Pic(M)$ contains an hyperbolic plane. Since the Picard rank is a derived invariant for K3 surfaces, the Picard group is an hyperbolic plane.
\end{proof}

\begin{proof}[Proof of the Main Theorem]
Consider $X$ any K3 surface, $\vv=(r,\Delta,s)$ a primitive vector with $\vv^{2} \geq -2$, and $\sigma \in \Stab^{\dagger}(X)$ a $\vv$-generic stability condition. The cases $\vv^2=-2$ and $\vv^2=0$ were proved in Theorem \ref{Yoshioka_semirigido}, so we assume $\vv^2 >0$. By Corollary \ref{Cor:projectivity} we see that $M_{\sigma}(\vv)$ is smooth, symplectic and projective. Since the Hodge numbers are invariant under deformations of projective varieties it is enough to show that $M_{\sigma}(\vv)$ is deformation equivalent (via a relative moduli space) to the Hilbert scheme of points on a K3 surface.

Applying Lemma \ref{change_vector} we can assume $r>0$ and $\Delta$ ample. Under this assumptions, using Lemma \ref{first_deformation} we deform to an elliptic K3 with a section and using Lemma \ref{coprime} we can assume $r$ and $\Delta$ to be coprime. Finally with Lemma \ref{fm} we reduce to $\vv=(1,0,1-n)$ with $n>1$, a generic $\sigma \in U(X)$, and $\Pic(X)=\Z \ee \oplus \Z f$.

Let $d=k^2(n-1)$, the class $s+(d+1)f$ is ample on $X$ of degree $2d$. With a small deformation we reduce to the case of a K3 surface $X'$ of Picard rank one, degree $2d$ and Mukai vector $\vv=(1,0,1-n)$. Since stability is an open property for families of objects, every skyscraper sheaf is still stable with respect to the deformed stability condition. So the deformed stability condition lies in the geometric chamber $U(X')$ by Lemma \ref{UandV}. By definition of $U(X')$ we can act by the group $\widetilde{GL_2(\R)}$ and get a stability condition $\sigma_{\alpha,\beta} \in V(X')$. This brings us in the setting of Corollary \ref{Cor:unigonal}, and we conclude the proof. Indeed, at every step of the reduction we get either isomorphisms that preserve the Mukai homomorphism by Proposition \ref{prop:equivalences}, or deformations that also preserve the Mukai homomorphism by Remark \ref{rem:defMukai}.
\end{proof}


\begin{thebibliography}{C-MTB11}

\bibitem[AP06]{AP:tstructures} Abramovich, D., Polishchuk, A., Sheaves of {$t$}-structures and valuative criteria for stable complexes, {\it J.~Reine Angew.~Math.} {\bf 590} (2006), 89--130.

\bibitem[BL+19]{bayer2019stability} Bayer, A., Lahoz, M., Macr\`i, E., Nuer, H., Perry, A., Stellari, P., Stability conditions in family, eprint  {\texttt {arXiv:1902.08184}}. 

\bibitem[BaMa11]{bayer_local} Bayer, A., Macr\`i, E., The space of stability conditions on the local projective plane, {\it Duke Math.\ J.}  {\bf 160} (2011), 263--322.

\bibitem[BaMa14a]{bayer_projectivity_2013} \bysame, Projectivity and birational geometry of Bridgeland moduli spaces, {\it J. Amer. Math. Soc.}  {\bf 27} (2014), 707--752.

\bibitem[BaMa14b]{mmp} \bysame, MMP for moduli of sheaves on K3s via wall-crossing: nef and movable cones, Lagrangian fibrations, {\it Invent.~Math.}  {\bf 198} (2014),  505--590

\bibitem[Bea83]{beauville_1983} Beauville, A., Vari\'et\'es K\"{a}hleriennes dont la premi\`ere classe de Chern est nulle, {\it J. Differential Geom.} {\bf 18} (1983), 755--782.

\bibitem[Bri07]{Bridgeland_triangulated} Bridgeland, T., Stability conditions on triangulated categories. {\it Ann. of Math. (2)} {\bf 166} (2007), 317--345.

\bibitem[Bri08]{bridgeland_stability_2006} \bysame, Stability conditions on K3 surfaces, {\it Duke Math.\ J.} {\bf 141} (2008), 241--291.

\bibitem[BBH]{bartocci_1997} Bartocci, C., Bruzzo, U. and Hern\'{a}ndez Ruip\'{e}rez, D., A Fourier-Mukai transform for stable bundles on K3 surfaces, {\it J. Reine Angew. Math.} {\bf 486} (1997), 1--16.

\bibitem[Har12]{hartmann_cusps_2012} Hartmann, H., Cusps of the K\"{a}hler moduli space and stability conditions on K3 surfaces, {\it Math. Ann.} {\bf 354} (2012), 1--42.

\bibitem[Huy97]{huybrechts_birational} Huybrechts, D., Birational symplectic manifolds and their deformations, {\it J. Diff. Geom.} {\bf 45} (1997), 488--513.

\bibitem[Huy06]{Huybrechts_FM} \bysame, Fourier-Mukai transforms in algebraic geometry, {\it Oxford Mathematical Monographs}, Oxford University Press, Oxford, 2006.

\bibitem[HL10]{huybrechts_lehn_2010} Huybrechts, D., Lehn, M., The Geometry of Moduli Spaces of Sheaves, {\it Cambridge Mathematical Library}, Cambridge University Press, Cambridge, 2010.

\bibitem[HS05]{Huybrechts_2005} Huybrechts, D., Stellari, P., Equivalences of twisted K3 surfaces, {\it Math. Ann.} {\bf 332} (2005), 901--936.

\bibitem[Ina02]{inaba_stable} Inaba, M., Toward a definition of moduli of complexes of coherent sheaves on a projective scheme, {\it J. Math. Kyoto Univ.} {\bf 42} (2002), 317--329.

\bibitem[Ina11]{inaba_smoothness_2010} \bysame, Smoothness of the moduli space of complexes of coherent sheaves on an abelian or a projective K3 surfaces, {\it Adv. Math.} {\bf 227} (2011), 1399--1412.

\bibitem[Kul90]{kuleshov90} Kuleshov, S., Stable bundles on a K3 surface, {\it Izv. Akad. Nauk SSSR Ser. Mat.} {\bf 54} (1990), 213--220

\bibitem[Lie06]{Lieblich_complexes} Lieblich, M., Moduli of complexes on a proper morphism, {\it J.\ Algebraic Geom.} {\bf 15} (2006), 175--206.

\bibitem[MS20]{macri2020stability} Macr\`i, E., Schmidt, B., Stability and applications, eprint {\texttt {arXiv:2002.01242}}.

\bibitem[Muk84]{mukai_symplectic_1984} Mukai, S., Symplectic structure of the moduli space of sheaves on an abelian or K3 surface, {\it Invent. Math.} {\bf 77} (1984), 101--116.

\bibitem[Muk87a]{Muk:K3} \bysame, On the moduli space of bundles on K3 surfaces. I, {\it Vector bundles on algebraic varieties} (Bombay, 1984), 341--413, Tata Inst. Fund. Res. Stud. Math. {\bf 11}, Tata Inst. Fund. Res., Bombay, 1987.

\bibitem[Muk87b]{Muk:FM} \bysame, Fourier functor and its application to the moduli of bundles on an Abelian variety, {\it Adv. Studies Pure Math.} {\bf 10} (1987), 515--550

\bibitem[O'G97]{ogrady_weight-two_1995} O'Grady, K., The weight-two Hodge structure of moduli spaces of sheaves on a K3 surface, {\it J. Algebraic Geom.} {\bf 6} (1997), 599--644.

\bibitem[PR18]{perego_moduli_2018} Perego, A., Rapagnetta, A., The moduli spaces of sheaves on K3 surfaces are irreducible symplectic varieties, eprint  {\texttt {arXiv:1802.01182}}. 

\bibitem[ST01]{seidel_2001} Seidel, P., Thomas, R., Braid group actions on derived categories of coherent sheaves, {\it Duke Math. J.} {\bf 108} (2001), 37--108.

\bibitem[Tod08]{Toda_Moduli} Toda, Y., Moduli stacks and invariants of semistable objects on K3 surfaces, {\it Adv. Math.} {\bf 217} (2008), 2736--2781.

\bibitem[Yos01]{Yoshioka_main} Yoshioka, K., Moduli spaces of stable sheaves on abelian surfaces, {\it Math. Ann.} {\bf 321} (2001), 817--884.

\bibitem[Yos03]{Yoshioka_FM1} \bysame, Stability and the Fourier-Mukai transform I, {\it Math. Z.} {\bf 245} (2003), 657--665.

\bibitem[Yos09]{Yoshioka_FM2} \bysame, Stability and the Fourier-Mukai transform II, {\it Comp. Math.} {\bf 145} (2009), 112--142.

\end{thebibliography}
\end{document}